\documentclass[12pt,reqno]{amsart}
\usepackage{amsmath,amsthm,amssymb,amsfonts,amscd}
\usepackage{mathrsfs}
\usepackage{bbm}
\usepackage{bbding}
\usepackage{graphicx,latexsym}
\usepackage[backref=page]{hyperref}
\usepackage{hyperref}
\usepackage{geometry}
\usepackage{color}
\usepackage{xcolor}

\usepackage{picture,epic}
\usepackage{tikz}

\usepackage{enumitem}

\numberwithin{equation}{section}

\theoremstyle{plain}
\newtheorem{theorem}{Theorem}[section]
\newtheorem{lemma}[theorem]{Lemma}

\newtheorem{proposition}[theorem]{Proposition}

\theoremstyle{definition}

\theoremstyle{remark}
\newtheorem{remark}[theorem]{Remark}

\renewcommand{\Re}{\operatorname{Re}}
\renewcommand{\Im}{\operatorname{Im}}

\newcommand{\sym}{\operatorname{sym}}

\newcommand{\GL}{\operatorname{GL}}

\renewcommand{\mod}{\operatorname{mod}\ }

\newcommand{\dd}{\mathrm{d}}

\newcommand{\holo}{\text{holo}}

\makeatletter
\def\@tocline#1#2#3#4#5#6#7{\relax
  \ifnum #1>\c@tocdepth 
  \else
    \par \addpenalty\@secpenalty\addvspace{#2}%
    \begingroup \hyphenpenalty\@M
    \@ifempty{#4}{%
      \@tempdima\csname r@tocindent\number#1\endcsname\relax
    }{%
      \@tempdima#4\relax
    }%
    \parindent\z@ \leftskip#3\relax \advance\leftskip\@tempdima\relax
    \rightskip\@pnumwidth plus4em \parfillskip-\@pnumwidth
    #5\leavevmode\hskip-\@tempdima
      \ifcase #1
       \or\or \hskip 1em \or \hskip 2em \else \hskip 3em \fi%
      #6\nobreak\relax
    \hfill\hbox to\@pnumwidth{\@tocpagenum{#7}}\par
    \nobreak
    \endgroup
  \fi}
\makeatother

\begin{document}

\title[Mixed moments of $L$-functions]
{Mixed moments of $\rm GL(2)$ and symmetric square $L$-functions}
\author{Bingrong Huang}
\author{Liangxun Li}
\address{Data Science Institute and School of Mathematics \\ Shandong University \\ Jinan \\ Shandong 250100 \\China}
\email{brhuang@sdu.edu.cn}
\email{lxli@mail.sdu.edu.cn}


\date{\today}

\begin{abstract}
  In this paper, we prove asymptotic formulas of mixed moments of  $\rm GL(2)$ and its symmetric square $L$-functions for both Hecke--Maass cusp forms and holomorphic Hecke eigenforms in short intervals.
  As an application, we prove quantitative simultaneous non-vanishing of central values of these $L$-functions.

\end{abstract}


\thanks{This work was supported by  the National Key R\&D Program of China (No. 2021YFA1000700) and NSFC (No. 12031008).}

\maketitle

\section{Introduction} \label{sec:Intr}

Estimating moments of $L$-functions is one of the main problems in number theory and has many applications, such as the non-vanishing problem   and the subconvexity problem of $L$-functions. See e.g. \cite{Duke1995,IS00Laudau-Siegel,CI,DFI2002,young2014weyl}.
Proving asymptotic formulas of moments of $L$-functions is particularly interesting as the main terms were conjectured by using the random matrix theory \cite{KS2000}, the theory of
multiple Dirichlet series \cite{DGH2003} or the recipe in \cite{CFKRS}.

Denote an orthonormal basis of even Hecke--Maass cusp forms for $SL_2(\mathbb{Z})$ by $\{u_j\}_{j\geq1}$.
Let $\Delta_{\mathbb{H}}=-y^{-2} (\partial^2/\partial x^2+\partial^2/\partial y^2)$ be the hyperbolic Laplace operator.
We have $\Delta_{\mathbb{H}} u_j = (1/4+t_j^2) u_j$, where $t_j>1$ is the spectral parameter of $u_j$. Let $\lambda_j(n)$ be $n$-th  Hecke eigenvalue of $u_j$.
The $L$-function of $u_j$ is defined by
\[
  L(s,u_j) = \sum_{n\geq1} \frac{\lambda_j(n)}{n^s}, \quad \Re(s)>1.
\]
The symmetric square $L$-function is defined by
\[
  L(s,\sym^2 u_j) = \zeta(2s) \sum_{n\geq1} \frac{\lambda_j(n^2)}{n^s}, \quad \Re(s)>1.
\]
Let $\Gamma_\mathbb{R}(s) = \pi^{-s/2} \Gamma(s/2)$,
$\gamma(s,u_j) =  \Gamma_\mathbb{R}(s+it_j) \Gamma_\mathbb{R}(s-it_j)$,
and $\gamma(s,\sym^2u_j) = \Gamma_\mathbb{R}(s) \Gamma_\mathbb{R}(s+2it_j) \Gamma_\mathbb{R}(s-2it_j)$.
Then $L(s,u_j)$ and $L(s,\sym^2 u_j)$ have analytic continuations to $\mathbb{C}$ and satisfy the functional equations
\[
  \Lambda(s,u_j) := \gamma(s,u_j) L(s,u_j) = \Lambda(1-s,u_j),
\]
and
\[
  \Lambda(s,\sym^2u_j) := \gamma(s,\sym^2u_j) L(s,\sym^2u_j) = \Lambda(1-s,\sym^2u_j).
\]
\subsection{The mixed moment}
In this paper, we are interested in the central values $L(1/2,u_j)$ and $L(1/2,\sym^2 u_j)$ when $u_j$ varies in the family of even Hecke--Maass cusp forms. Note that for odd forms $u$, we have $L(1/2,u)=0$.
It is natural to have harmonic weights when one considers the averages of central values. We introduce the harmonic weights
\[
  \omega_j = \frac{2\pi}{L(1,\sym^2 u_j)}.
\]
By Iwaniec \cite{iwaniec1990small} and Goldfeld--Hoffstein--Lieman \cite{Hoffstein-Lockhart}, we have
\begin{equation}\label{eqn: omega^*_j}
  (\log t_j)^{-1} \ll \omega_j  \ll  {t_j}^{\varepsilon}.
\end{equation}
For positive numbers $T, \Delta\geq 1$, let
\[
   h_{T,\Delta}(t)= e^{-\frac{(t-T)^2}{\Delta^2}} + e^{-\frac{(t+T)^2}{\Delta^2}}.
\]
Our main result in this paper is the following asymptotic formula.

\begin{theorem}\label{thm:moment}
  Let $T^\varepsilon \leq \Delta \leq T^{1-\varepsilon}$.
  Let $\{u_j\}_{j\geq1}$ an orthonormal basis of even Hecke--Maass cusp forms for $\rm SL_2(\mathbb{Z})$. Denote $t_j>1$ be the spectral parameter of $u_j$.
  Then we have
  \begin{equation}\label{eqn: mixedmoment}
     \sideset{}{'}\sum_{j\geq1} w_j h_{T,\Delta}(t_j) L(1/2,u_j) L(1/2,\sym^2 u_j) =a_1 H_{T,\Delta}^{\log}+a_2H_{T,\Delta}+O(T^{1+\varepsilon}+\Delta T^{\frac{3}{4}+\varepsilon}),
  \end{equation}
  where $\sum'$ restricts to the even Hecke--Maass cusp forms,
  \[
        H_{T,\Delta}^{\log}=\int_{0}^{\infty}e^{-\frac{(t-T)^2}{\Delta^2}}t\log t\,\dd t,
  \]
  \[
        H_{T,\Delta}=\int_{0}^{\infty}e^{-\frac{(t-T)^2}{\Delta^2}}t\,\dd t,
  \]
  $
        a_1=\frac{2}{\pi}\zeta(\frac{3}{2})
  $
  and
  $
        a_2= \frac{4}{\pi}\left(\zeta(\frac{3}{2})\left(\frac{3\gamma}{4}-\frac{3}{4}\log \pi -\frac{3\log 2}{4}-\frac{\pi}{8}\right)+\frac{1}{2}\zeta'(\frac{3}{2})\right).
  $
\end{theorem}

Here  the main term of the mixed moment is $H_{T,\Delta}^{\log}$. By direct calculations, we have
  \begin{equation*}
        H_{T,\Delta}^{\log}=\sqrt{\pi}\Delta T\log T+O(\Delta^{2+\varepsilon}),
  \end{equation*}
  and
   \begin{equation*}
        H_{T,\Delta}=\sqrt{\pi}\Delta T +O(T^{-N}),
  \end{equation*}
for any positive integer $N$.
Therefore the mixed moment \eqref{eqn: mixedmoment} is equal to
\[
    a_1\sqrt{\pi} T\Delta\log T+a_2 \sqrt{\pi}T\Delta+O(T^{1+\varepsilon}+\Delta T^{\frac{3}{4}+\varepsilon}+\Delta^{2+\varepsilon}).
\]
\par


Moments of $\GL(2)$ $L$-functions in the spectral aspect have been studied extensively. For example, Ivi\'c  \cite{Ivic2002} and Qi \cite{Q22Cubicmoment} proved the asymptotic formulas for the cubic moments. It is natural to estimate moments of the symmetric square $L$-functions. See e.g. Balkanova--Frolenkov \cite{BF2021} for the first and second moments.
It is interesting to study the joint value distribution of the $\GL(2)$ $L$-functions and their symmetric square $L$-functions.
In Theorem \ref{thm:moment}, we consider the first moment of these $L$-functions.
To the best of your knowledge, there is no such a result in the spectral aspect.
However, Balkanova,  Bhowmik, Frolenkov, and Raulf did a series of work about the mixed moments for  holomorphic Hecke eigenforms in the weight aspect, see \cite{BBFR2019,BBFR2020}. (We will discuss this in more details in \S \ref{subsec:weight}). In the level aspect,
Munshi and Sengupta  \cite{MS2018simultaneousnon-vanishing} proved an asymptotic formula for the mixed moment of $L(1/2, f)L(1/2, \sym^2f)$.\par

In the proof of Theorem \ref{thm:moment}, after using the Kuznetsov trace formula and the Poisson summation formula, a natural way to give bounds for the off-diagonal term is applying an integrated large sieve inequality. In our paper,
an extended large sieve inequality is helpful. Specifically,
Let $a_{m,n}$ be supported on primitive points, i.e. $a_{m,n}=0$ for $(m,n)\neq 1$. We have
\begin{equation}\label{eqn: LSI}
        \int_{|u|\leq U}\left|\sum_{n\leq N}\sum_{m\leq M}a_{m,n}\left(\frac{m}{n}\right)^{iu}\right|^2\dd u\ll (MN)^\varepsilon(U+MN)\sum_{m\leq M}\sum_{n\leq N}|a_{m,n}|^2.
\end{equation}
It is roughly the form of the classical large sieve inequality with replacing the single summation by the double summations.
Moreover, this inequality can help us to handle the case when the inner double sums  can not be separated.
According to \cite[Chapter 7.3]{IwaniecKowalski2004analytic},
the bound  in \eqref{eqn: LSI} is essentially best possible.

\subsection{The non-vanishing problem}
The non-vanishing problems for central values for $\rm GL(2)$ and symmetric square $L$-functions have studied a lot in the past decades.
For the $\rm  GL(2)$ family, a large number of non-vanishing results have been established in the spectral aspect, see e.g. \cite{Liu18Nonvanishing,BHS2021Nonvanishing};  in the weight aspect, see e.g. \cite{Fomenko00, LT05, Luo2015}; in the level aspect, see e.g. \cite{Duke1995, IS00Laudau-Siegel}. In order to yield an effective non-vanishing results, the combination of the moment method and modification method plays a significant role in this area.
Such non-vanishing results have various applications.
A remarkable work due to Iwaniec and Sarnak \cite{IS00Laudau-Siegel}, they showed that for large $N$, at least $50\%$ of $L(1/2,f)$ with even holomorphic form $f$ for  $\Gamma_0(N)$ are positive,
and any improvement of $50\%$ will imply the nonexistence of Landau--Siegel zeros.
For the symmetric square family, Khan \cite{Khan2012Nonvanishing} established a positive proportional non-vanishing result for $L(1/2, \sym^2 f)$ with holomorphic form $f$ in the range of weight $k\asymp K$. The result for short intervals is obtained by Balkonova and Frolenkov \cite{BF2021}.

It is a natural question to ask whether there are infinitely many holomorphic Hecke eigenforms or Hecke--Maass cusp forms $\phi$ such that the central $L$-values of $\sym^r \phi$, $r=1,2,\ldots$, are simultaneous non-vanishing.
In this paper, we focus on the simplest case on simultaneous non-vanishing of central $L$-values of $\phi$ and $\sym^2 \phi$. From the result in \cite{BBFR2019}, we can know that there are infinitely many Hecke--Maass cusp forms $\phi$  such that   $L(1/2,\phi)$ and  $L(1/2,\sym^2 \phi)$ are simultaneous non-vanishing.
It is harder  to prove simultaneous non-vanishing result locally.
We can prove that, for any sufficiently large $T$ and any $\varepsilon>0$,  there are more than $T^{2/15}$ Hecke--Maass cusp forms $\phi$ with $|t_\phi-T|\leq T^{\varepsilon}$ such that   $L(1/2,\phi)$ and  $L(1/2,\sym^2 \phi)$ are simultaneous non-vanishing.
More precisely, we have the following quantitative result.

\begin{theorem}\label{thm:non-vanishing}
  Let $T^\varepsilon \leq \Delta \leq T^{1-\varepsilon}$. Then we have
  \[
     \sideset{}{'}\sum_{\substack{T-\Delta \leq  t_j \leq T+\Delta\\ L(1/2,u_j) L(1/2,\sym^2 u_j)\neq 0 }} 1 \gg T^{2/15-\varepsilon} \Delta^2,
  \]
  if $T^\varepsilon \leq \Delta \leq T^{1/5+\varepsilon}$; and
  \[
     \sideset{}{'}\sum_{\substack{T-\Delta\leq  t_j \leq T+\Delta \\ L(1/2,u_j) L(1/2,\sym^2 u_j)\neq 0 }} 1 \gg T^{1/3-\varepsilon}\Delta,
  \]
  if $T^{1/5+\varepsilon} \leq \Delta \leq T^{1-\varepsilon}$.
\end{theorem}
\begin{proof}[Proof of Theorem \ref{thm:non-vanishing} under Theorem \ref{thm:moment}]
Theorem \ref{thm:moment} implies that
\begin{equation}\label{eqn: mixedmoment_size}
    \sideset{}{'}\sum_{T-\Delta\leq t_j \leq T+\Delta}L(1/2, u_j)L(1/2,\sym^2 u_j)= T^{1+o(1)}\Delta.
 \end{equation}
In order to prove Theorem \ref{thm:non-vanishing}, we will use the Weyl-type bound $L(1/2,u_j)\ll t_j^{\frac{1}{3}+\varepsilon}$ and Khan and Young's second moment estimate \cite{Khan-Young}: for $\Delta\geq T^{1/5+\varepsilon}$,
\begin{equation}\label{eqn:KhanYoung}
    \sideset{}{'}\sum_{T-\Delta\leq t_j \leq T+\Delta}|L(1/2,\sym^2 u_j)|^2\ll T^{1+\varepsilon}\Delta.
 \end{equation}
For $T^\varepsilon\leq \Delta\leq T^{1/5+\varepsilon}$, the above bound gives
\begin{equation}\label{eqn:secondmoment}
    \sideset{}{'}\sum_{T-\Delta\leq t_j \leq T+\Delta}|L(1/2,\sym^2 u_j)|^2
    \ll \sideset{}{'}\sum_{T-T^{1/5+\varepsilon}\leq t_j \leq T+T^{1/5+\varepsilon}}|L(1/2,\sym^2 u_j)|^2
    \ll T^{6/5+\varepsilon}.
 \end{equation}
Applying the Cauchy--Schwarz inequality, we have
\begin{align*}
    \bigg(\sideset{}{'}\sum_{T-\Delta\leq t_j \leq T+\Delta} & L(1/2, u_j)L(1/2,\sym^2 u_j) \bigg)^2 \\
   & \ll
    \sideset{}{'}\sum_{\substack{T-\Delta \leq  t_j \leq T+\Delta\\ L(1/2,u_j) L(1/2,\sym^2 u_j)\neq 0 }} 1 \cdot \sideset{}{'}\sum_{T-\Delta\leq t_j \leq T+\Delta}|L(1/2, u_j)|^2|L(1/2,\sym^2 u_j)|^2 \\
   & \ll T^{2/3+\varepsilon}
    \sideset{}{'}\sum_{\substack{T-\Delta \leq  t_j \leq T+\Delta\\ L(1/2,u_j) L(1/2,\sym^2 u_j)\neq 0 }} 1
    \cdot
    \sideset{}{'}\sum_{T-\Delta\leq t_j \leq T+\Delta}|L(1/2,\sym^2 u_j)|^2.
\end{align*}
Combining  \eqref{eqn: mixedmoment_size} and the above bounds, we complete the proof.
\end {proof}
Our quantitative result may be improved by establishing Lindel\"{o}f-on-average bound (\ref{eqn:KhanYoung}) in shorter intervals,
especially $\Delta=1,$ or obtaining the sub-Weyl bound for $GL(2)$ $L$-function on spectral aspect. Both of them are standing challenges in the theory of $L$-functions.


\subsection{The holomorphic case}\label{subsec:weight}
Let $H_{4k}$ be the normalised Hecke basis for the space of holomorphic cusp forms of level $1$ and weight $4k$ where $k\geq 1$.
Balkanova,  Bhowmik, Frolenkov and Raulf \cite{BBFR2020} consider the mixed moment
\begin{equation}\label{mixedmoment_holo}
\mathcal{M}_{0,0}:=\sum_{f\in H_{4k}}w_f L(1/2, f)L(1/2, \sym^2f).
\end{equation}
Notice that if the weight is not divided by $4$, the functional equation forces $L(1/2, f)=0$.
Here the harmonic weight $w_f=\frac{2\pi^2}{(4k-1)L(1,\sym^2f)}$  which satisfies $\sum_{f\in H_{4k}}w_f\asymp 1$,
so the expected main term has size $k^{o(1)}$.
Firstly, they established the asymptotic when they sum over the weights $k \asymp K$ \cite{BBFR2019}.
In \cite[Theorem 1.1]{BBFR2020},
they found an explicit formula for $\mathcal{M}_{0,0}$
\begin{equation}\label{eqn: M(0,0)}
\mathcal{M}_{0,0}=2\mathcal{M}_{0,0}^D+2\mathcal{M}_{0,0}^{OD}+\frac{1}{2\pi i}\int_{(0)}G_{2k}(0,s)\dd s+O(k^{-1+\varepsilon}).
\end{equation}
Each of the terms above has an explicit expression.
In particular,
\begin{equation*}
\mathcal{M}_{0,0}^D=\zeta(3/2)\log k+\frac{\zeta(3/2)}{2}\left(\frac{\pi}{2}-3\log(2\pi)+3\gamma+\frac{2\zeta'(3/2)}{\zeta(3/2)}+2\log 2\right)+O(k^{-1}).
\end{equation*}
is the main term. Here $\gamma$ is the Euler constant,
 $\mathcal{M}_{0,0}^{OD}$ is of size $k^{-1/2}$, and
$G_{2k}(0,s)$ has a complicated expression involving $\Gamma$-functions, double Dirichlet series and hypergeometric function ${}_3F_2$.
The analysis of the third term which is given by the integral of $G_{2k}(0,s)$ is the core of their work. They can prove (see \cite[Eq. (6.59)]{BBFR2020})
\begin{equation}\label{eqn: intG_bound}
    \frac{1}{2\pi i}\int_{(0)}G_{2k}(0,s)\dd s\ll \log^3 k.
\end{equation}
However this is not good enough to produce an asymptotic formula for  $\mathcal{M}_{0,0}$, which seems to be very difficult.
In order to prove an asymptotic formula for the mixed moment,
we consider the mixed moment (\ref{mixedmoment_holo}) with an extra average over the weights in short intervals.
By using the same method as in the Maass form case, we have the following asymptotic formula.
\begin{theorem}\label{thm:moment_holo}
  Let $K^\varepsilon \leq \Delta \leq K^{1-\varepsilon}$ and  $h$ is a smooth compact supported function.
  Then we have
  \begin{multline}
     \sum_{k\geq 1}h\left(\frac{4k-K-1}{\Delta}\right)\sum_{f\in H_{4k}}w_f L(1/2, f)L(1/2, \sym^2f)\\
     =a_3\sum_{k\geq 1}h\left(\frac{4k-K-1}{\Delta}\right)\log k
     +a_4\sum_{k\geq 1}h\left(\frac{4k-K-1}{\Delta}\right)+O(K^\varepsilon+\Delta K^{-\frac{1}{4}+\varepsilon}),
  \end{multline}
  where
    $a_3=2\zeta(3/2),$
  and
    $a_4=\zeta(3/2)\left(\frac{\pi}{2}-3\log(2\pi)+3\gamma+\frac{2\zeta'(3/2)}{\zeta(3/2)}+2\log 2\right).$
\end{theorem}
\begin{remark}
From \eqref{eqn: M(0,0)}, providing that $\frac{1}{2\pi i}\int_{(0)}G_{2k}(0,s)\dd s=o(\log k)$, we can establish the asymptotic formula for $\mathcal{M}_{0,0}$. As we expect, this term should be small.
Comparing $\eqref{eqn: M(0,0)}$ with Theorem \ref{thm:moment_holo}, we have
\begin{equation}
\sum_{k\geq 1}h\left(\frac{4k-K-1}{\Delta}\right)\frac{1}{2\pi i}\int_{(0)}G_{2k}(0,s)\dd s\ll K^\varepsilon+\Delta K^{-\frac{1}{4}+\varepsilon}.
\end{equation}
This shows that by adding a short average, the integral of $G_{2k}(0,s)$ is small.
\end{remark}

\subsection{Overview}
We now give a rough sketch of the proof for Theorem \ref{thm:moment}.
The basic idea is similar to Conrey and Iwaniec \cite{CI} and Young \cite{young2014weyl}.
Using the approximate functional equations of $L(1/2, u_j)$ and $L(1/2, \sym^2 u_j)$ and the Kuznetsov trace formula, we get roughly
  \begin{multline*}
     \sideset{}{'}\sum_{j\geq1}w_j h_{T, \Delta}(t_j) L(1/2,u_j) L(1/2,\sym^2 u_j)+(\text{Est.})\\
     =(\text{D. T.})
     +2\sum_{m\leq T^{1+\varepsilon}}\sum_{n\leq T^{1+\varepsilon}}\frac{1}{\sqrt{mn}} \sum_{c\geq1}\frac{S(n^2,\pm m; c)}{c} H^{\pm}\left(\frac{4\pi\sqrt{m}n}{c}\right),
  \end{multline*}
where $(\text{Est.})$ represents the contribution of continuous spectrum, and $(\text{D. T.})$ represents  the contribution of diagonal term which is the main term of the mixed moment.
By the Cauchy--Schwarz inequality, $(\text{Est.})$ is essentially bounded by the fourth moment of Riemann zeta function ($\ll T^{1+\varepsilon}$).
$(\text{D. T.})$ can be evaluated by shifting the lines of integrals and collecting the contributions of the residues.
The off-diagonal here contributes the error term which is different from the case of the $GL(2)$ cubic moment (See e.g. \cite{Q22Cubicmoment}. After applying the Poisson summation formula, the zero-frequency of the off-diagonal terms contributes a part of the main term).
With cleaning up some error terms, taking a smooth dyadic partition of unity, we need to bound the following sum
\[
\mathcal{S}_{\pm}(M, N; C)=\sum_{m\asymp M}
  \sum_{n\asymp N}\sum_{c\asymp C}\frac{S(n^2,\pm m;c)}{c} H^{\pm}\left(\frac{4\pi\sqrt{m}n}{c}\right),
\]
where $M, N$ are dyadic parameters less than $T^{1+\varepsilon}$, and $C$ less than a fixed power of $T$.
Our aim is to show
\[
\mathcal{S}_{\pm}(M, N; C)\ll T^{1+\varepsilon}\sqrt{MN}.
\]\par
  Following \cite{CI}, applying the Poisson summation formula to $m,n$-sum, we obtain the dual form
\begin{equation*}
     \mathcal{S}_{\pm}(M, N; C)=\sum_{c\asymp C}\mathop{\sum\sum}_{x_1, x_2\in \mathbb{Z}}W_{\pm}(x_1, x_2, c)G_{\pm}(x_1, x_2, c),
\end{equation*}
where
\begin{equation*}
     W_{\pm}(x_1, x_2; c )=\int_{t_1\asymp M}\int_{t_2\asymp N}H^\pm\left( \frac{4\pi \sqrt{t_1}t_2}{c}\right)e\left(-\frac{x_1t_1}{c}-\frac{x_2t_2}{c}\right) \dd t_1 \dd t_2
\end{equation*}
and
\begin{equation*}
     G_\pm(x_1, x_2, c)=\frac{1}{c^3}\sum_{a, b\bmod c}S(\pm a, b^2; c)e\left(\frac{ax_1}{c}+\frac{bx_2}{c}\right).
\end{equation*}
  We split the sum into two parts, the oscillatory case of $|x_1|M/C\geq T^\varepsilon$ or $|x_2|N/C\geq T^{\varepsilon}$, and the non-oscillatory case of $|x_1|M/C\leq T^\varepsilon$ and $|x_2|N/C\leq T^{\varepsilon}$.
  The second case is easy to estimate, see \S \ref{subsection: remaining-case}. For the first case,  using the Young's techniques such as the stationary phase method and the Mellin transform to analyze the analytic part  $W_{\pm}(x_1, x_2; c )$, we get
\begin{equation}
W_{\pm}(x_1, x_2; c )=\frac{cM\Delta}{x_2}e\left(\mp\frac{x_1{x_2}^2}{4c}\right)\int_{|u|\ll U_{\pm}}\lambda(u)\left(\frac{x_1{x_2}^2}{4c}\right)^{iu}\dd u+O(T^{-10}).
\end{equation}
where $\lambda(u)\ll 1$ and $U_{\pm}\ll\frac{T^{1+\varepsilon}}{\Delta}$, for some essential range of $x_1, x_2$, see Lemma \ref{lemma:Wproperty}.
\par
  To calculate the arithmetic part $G_{\pm}$, opening the Kloosterman sum, by the orthogonality of additive characters, we get
  a complete exponential sum of quadratic polynomial. And we can show that $G_{\pm}$ has the form (see (\ref{eqn: G=G_1,2}))
\begin{equation}
G_{\pm}(x_1, x_2; c )=e\left(\pm\frac{x_1{x_2}^2}{4c}\right)\sum_{i=1, 2} G_{i}(x_1; c)a(x_1)b(x_2).
\end{equation}
where $a(x_1), b(x_2)\ll 1$, $G_{i}(x_1; c)$  is supported on $(x_1, c)=1$ and satisfies the second moment estimate
\begin{equation*}
\sum_{c\asymp C}|G_i(x_1; c)|^2\ll \frac{1}{C^2},
\end{equation*}\par
  Combining the analytic part and the arithmetic part, letting  $x_1\asymp X_1$ and $x_2\asymp X_2$, we get the total contribution is bounded by
\begin{equation}
    \mathop{\sum\sum}_{X_1, X_2(\star)}\frac{CM\Delta}{X_2}
     \int_{|u|\ll U_{\pm}}\left|\sum_{c\asymp C}\sum_{x_1\asymp X_1}\sum_{x_2\asymp X_2}G_{i}(x_1; c)a(x_1)b(x_2)\lambda(u)\left(\frac{x_1{x_2}^2}{4c}\right)^{iu}\right|\dd u.
\end{equation}
 where the $\star$ means some restrictions on dyadic parameters $X_1, X_2$. Here we can't separate variables $x_1$ and $c$ in the arithmetic part. However we can prove the large sieve inequality in two dimension, see Lemma \ref{lemma:Largesieve2}.  The estimate in Lemma \ref{lemma:Largesieve2} requires that the array supports on coprime lattice, and $G_{i}$ just meets the requirement.
\par
 Applying the Cauchy--Schwarz inequality and the large sieve inequalities,  we get the bound
 \[
        \Delta \sqrt{MN}(U_{\pm}+X_1C)^{\frac{1}{2}}(U_{\pm}+X_2)^{\frac{1}{2}}T^\varepsilon.
 \]
An easy calculation shows this is bounded by $T^{1+\varepsilon}\sqrt{MN}$. Therefore the total off-diagonal term is bounded by $T^{1+\varepsilon}$.

\subsection{Plan for this paper}
The rest of this paper is organized as follows.
In \S \ref{sec:preliminaries}, we give a review of the theory of automorphic forms and $L$-functions, the Kuznetsov trace formula and some useful techniques. In \S \ref{sec:appyingKTT}, we begin our proof with applying the trace formula.  In \S \ref{sec:diagonal}, we calculate the diagonal terms as the main term of our mixed moment. In \S \ref{sec:offdiagonal}, we deal with the contribution of the off-diagonal term including the analysis of the analytic parts and the arithmetic part, application of the large sieve inequality and the estimation of the remaining case. This ends the proof of Theorem \ref{thm:moment}. In \S \ref{sec:holo}, we turn to the holomorphic case and give a sketch of proof of Theorem \ref{thm:moment_holo}, since it is roughly repeating the progress of \S \ref{sec:diagonal}.

\medskip
\textbf{Notation.}
Throughout the paper, $\varepsilon$ is an arbitrarily small positive number;
all of them may be different at each occurrence.
As usual, $e(x)=e^{2\pi i x}$.
We use $y\asymp Y$ to mean that $c_1 Y\leq y\leq c_2 Y$ for some positive constants $c_1$ and $c_2$.
The notation $x\lll y$ means we have $x\leq cy$ for some sufficiently small constant $c>0$,
and $x\ggg y$ means we have $x\geq Cy$ for some  sufficiently large  constant $C>0$

\section{Preliminaries}\label{sec:preliminaries}

\subsection{Gamma function and polygamma function}
We introduce some special functions for future use.
Let $\Gamma(z)$ be the Gamma function. The Stirling's formula gives
\[
  \Gamma(z) = z^{z-1/2} e^{-z} \sqrt{2\pi} \left( 1+\frac{1}{12 z} +\frac{1}{288 z^2} +O\left( \frac{1}{|z|^3} \right) \right).
\]
Let $\psi(z)$ be the polygamma function that is given by
\[
    \psi(z)=\frac{\Gamma'(z)}{\Gamma(z)}.
\]
It satisfies (see \cite[(8.365), (8.366)]{Table_of})
\[
    \psi(3/4-n)=\psi(1/4+n)+\pi,
\]
for any integer $n$ and
\[
    \psi(x+1)=\frac{1}{x}+\psi(x).
\]
Moreover,
\begin{equation}\label{eqn: psi1/4}
    \psi(1/4)=-\gamma-\frac{\pi}{2}-3\log 2,
\end{equation}
and
\begin{equation}\label{eqn: psi3/4}
    \psi(3/4)=-\gamma+\frac{\pi}{2}-3\log 2,
\end{equation}
 where $\gamma$ is the Euler constant .

\subsection{Approximate functional equations}
Let $E_t(z)=E(z,1/2+it)$ be the Eisenstein series and $t\in \mathbb{R}$, we have
\[
    L(s, E_t)=\zeta(s+it)\zeta(s-it),
\]
and
\[
    L(s, \sym^2E_t)=\zeta(s)\zeta(s+2it)\zeta(s-2it).
\]
Both of them have analytic continuations to $\mathbb{C}$ and satisfy the functional equations
\[
  \Lambda(s,E_t) := \gamma(s,E_t) L(s,u_j) = \Lambda(1-s,E_t),
\]
and
\[
  \Lambda(s,\sym^2E_t) := \gamma(s,\sym^2E_t) L(s,\sym^2E_t) = \Lambda(1-s,\sym^2E_t).
\]
where
$\gamma(s,E_t) =  \Gamma_\mathbb{R}(s+it) \Gamma_\mathbb{R}(s-it)$,
and $\gamma(s,\sym^2E_t) = \Gamma_\mathbb{R}(s) \Gamma_\mathbb{R}(s+2it) \Gamma_\mathbb{R}(s-2it)$.
\begin{lemma}\label{lemma:AFE2}
  Let $u_j$ be as in Theorem \ref{thm:moment}. Then we have
  \[
    L(1/2,u_j) = 2\sum_{m\geq1} \frac{\lambda_j(m)}{m^{1/2}} W\left(m,t_j\right)
  \]
  and
   \[
    L(1/2, E_t)=|\zeta(1/2+it)|^2 = 2\sum_{m\geq1} \frac{\eta_t(m)}{m^{1/2}} W\left(m,t\right),
  \]
  where
  \[
    W(m,t_j) = \frac{1}{2\pi i} \int_{(2)} \frac{\Gamma_\mathbb{R}(1/2+s+it_j)\Gamma_\mathbb{R}(1/2+s-it_j)} {\Gamma_\mathbb{R}(1/2+it_j)\Gamma_\mathbb{R}(1/2-it_j)}
    \frac{1}{m^s} G(s) \frac{\dd s}{s},
  \]
  and $W(m,t)$ is defined as above with $E_t$ replacing $u_j$.
   It is an even function on $t_j$ (or $t$), $G(s)=e^{s^2}$ and $\eta_t(m)=\sum_{ab=m}(\frac{a}{b})^{it}$.
  Let $T\geq10$ be large. For $t>0$ and $t\asymp T$, we have
  \[
    W(m,t) \ll_A \left(1+\frac{m}{T}\right)^{-A}, \quad \textrm{for any } A>0.
  \]
  Moreover, we have
  \[
    W(m,t) = \frac{1}{2\pi i} \int_{\varepsilon-i(\log T)^2}^{\varepsilon+i(\log T)^2} \left(\frac{t}{2\pi m}\right)^{s}  G(s) \frac{\dd s}{s} + O(T^{-2+\varepsilon}).
  \]
\end{lemma}

\begin{proof}
  The proof is the same as Iwaniec--Kowalski \cite[Theorem 5.3 and Proposition 5.4]{IwaniecKowalski2004analytic}.
  The last statement is a consequence of Stirling's formula.
\end{proof}

\begin{lemma}\label{lemma:AFE3}
  Let $u_j$ be as in Theorem \ref{thm:moment}. Then we have
  \[
    L(1/2,\sym^2 u_j) = 2\sum_{n\geq1} \frac{\lambda_j(n^2)}{n^{1/2}} V\left(n,t_j\right),
  \]
  and
  \[
    L(1/2,\sym^2 E_t)=\zeta(1/2)|\zeta(1/2+2it)|^2 = 2\sum_{n\geq1} \frac{\eta_t(n^2)}{n^{1/2}} V\left(n,t\right),
  \]
  where
  \[
    V(n,t_j) = \frac{1}{2\pi i} \int_{(2)} \frac{\Gamma_\mathbb{R}(1/2+s+2it_j)\Gamma_\mathbb{R}(1/2+s)\Gamma_\mathbb{R}(1/2+s-2it_j)} {\Gamma_\mathbb{R}(1/2+2it_j)\Gamma_\mathbb{R}(1/2)\Gamma_\mathbb{R}(1/2-2it_j)} \frac{\zeta(1+2s)}{n^s} G(s) \frac{\dd s}{s}
  \]
  and $V(n,t)$ is defined as above with $E_t$ replacing $u_j$.
  It is an even function on $t_j$ (or $t$).
  Let $T\geq10$ be large. For $t>0$ and $t\asymp T$, we have
  \[
    V(n,t) \ll_A \left(1+\frac{n}{T}\right)^{-A} \log T, \quad \textrm{for any } A>0.
  \]
  and $V(n,t)$ is defined as above with $E_t$ replacing $u_j$.
  Moreover, we have
  \[
    V(n,t) = \frac{1}{2\pi i} \int_{\varepsilon-i(\log T)^2}^{\varepsilon+i(\log T)^2} \left(\frac{t}{\pi n}\right)^s \frac{\Gamma_\mathbb{R}(1/2+s)} { \Gamma_\mathbb{R}(1/2)} \zeta(1+2s) G(s) \frac{\dd s}{s} + O(T^{-2+\varepsilon}).
  \]
\end{lemma}

\begin{proof}
  The proof is the same as Iwaniec--Kowalski \cite[Theorem 5.3 and Proposition 5.4]{IwaniecKowalski2004analytic}.
  The last statement is a consequence of Stirling's formula.
\end{proof}
For the convenience, we refine $W(m, t)$ and $V(n,t)$ by approximating $t$ by $T$.
Since that the weight function $h_{T, \Delta}(t)$ concentrates on $|t-T|\ll \Delta \log^2 T$,
we expanse the factor $t^s$ into Taylor series, and get
\[
t^s=T^s\sum_{l=0}^{N}P_l(s)\left(\frac{t^2-T^2}{T^2}\right)^l+O\left((1+\frac{|s|}{2})^{N+1}\left(\frac{t^2-T^2}{T}\right)^N\right).
\]
for a certain polynomial $P_l(s)$ of degree $\leq l$. Define
\[
 W_l(m, T)=\frac{1}{2\pi i} \int_{\varepsilon-i(\log T)^2}^{\varepsilon+i(\log T)^2} \left(\frac{T}{2\pi m}\right)^s P_l(s) G(s) \frac{\dd s}{s},
\]
and
\[
 V_l(n, T)=\frac{1}{2\pi i} \int_{\varepsilon-i(\log T)^2}^{\varepsilon+i(\log T)^2} \left(\frac{T}{\pi n}\right)^s \frac{\Gamma_\mathbb{R}(1/2+s)} { \Gamma_\mathbb{R}(1/2)} \zeta(1+2s)P_{l}(s) G(s) \frac{\dd s}{s}.
\]
We have
\[
   W_l(m, T)\ll\left(1+\frac{m}{T}\right)^{-A} \quad \text{and} \quad V_l(n, T)\ll\left(1+\frac{n}{T}\right)^{-A}\log T,
\]
\begin{equation}\label{eqn:Wt=WT}
W(m,t)=\sum_{l=0}^{N}W_l(m,T)\left(\frac{t^2-T^2}{T^2}\right)^{l}+O\left(\left(1+\frac{m}{T}\right)^{-A}\left(\frac{t^2-T^2}{T^2}\right)^{N+1}\right),
\end{equation}
and
\begin{equation}\label{eqn:Vt=VT}
V(n,t)=\sum_{l=0}^{N}V_l(n,T)\left(\frac{t^2-T^2}{T^2}\right)^{l}+O\left(\left(1+\frac{n}{T}\right)^{-A}\left(\frac{t^2-T^2}{T^2}\right)^{N+1}\log T\right).
\end{equation}

\subsection{The Kuznetsov trace formulas }
%


Let
\[
  S(a,b;c) = \sideset{}{^*}\sum_{d \bmod c} e\left(\frac{ad+b\bar{d}}{c}\right)
\]
be the classical Kloosterman sum. For any $m,n\geq1$, and any test function
$h(t)$ which is even and satisfies the following conditions:
\begin{itemize}
  \item [(i)] $h(t)$ is holomorphic in $|\Im(t)|\leq 1/2+\varepsilon$,
  \item [(ii)] $h(t)\ll (1+|t|)^{-2-\varepsilon}$ in the above strip,
\end{itemize}
we have the following Kuznetsov formula (see Conrey--Iwaniec \cite[Eq. (3.17)]{CI} for example).
\begin{lemma}\label{lemma:KTF}
  For $m,n\geq1$,   we have
  \begin{equation*}
    \begin{split}
        & {\sum_j}'h(t_j)\omega_j \lambda_j(m)\lambda_j(n) + \frac{1}{4\pi}\int_{-\infty}^{\infty}h(t)\omega(t)\eta_t(m)\eta_t(n)
        \dd t \\
        & \hskip 120pt = \frac{1}{2}\delta_{m,n}H
        + \frac{1}{2} \sum_{\pm} \sum_{c\geq1}\frac{S(n,\pm m;c)}{c} H^{\pm}\left(\frac{4\pi\sqrt{mn}}{c}\right),
    \end{split}
  \end{equation*}
  where $\sum'$ restricts to the even Hecke--Maass cusp forms, $\delta_{m,n}$ is the Kronecker symbol,
  \begin{equation}\label{eqn: H}
    \begin{split}
       H & = \frac{2}{\pi}\int_{0}^{\infty} h(t) \tanh(\pi t)t \dd t, \\
       H^+(x) & = 2i \int_{-\infty}^{\infty} J_{2it}(x)\frac{h(t)t}{\cosh(\pi t)} \dd t, \\
       H^-(x) & = \frac{4}{\pi} \int_{-\infty}^{\infty} K_{2it}(x)\sinh(\pi t)h(t)t \dd t,
    \end{split}
  \end{equation}
  and $J_\nu(x)$ and $K_\nu(x)$ are the standard $J$-Bessel function and $K$-Bessel function respectively.
\end{lemma}

To deal with $H^\pm(x)$ we will use the following lemmas.

\begin{lemma}\label{lemma:H+}
  Let $H^+$ be given by \eqref{eqn: H} with $h(t)= h_{T,\Delta}(t)$.
  There exists a function $g$ depending on $T$ and $\Delta$ satisfying $g^{(j)}(y)\ll_{j,A} (1+|y|)^{-A}$,
  so that
  \begin{equation}\label{eqn:H^+=int}
    H^+(x) = \Delta T\int_{|v|\leq \frac{\Delta^\varepsilon}{\Delta}} \cos(x\cosh(v))e\left(\frac{vT}{\pi}\right)g(\Delta v) \dd v + O(T^{-A}).
  \end{equation}
  Furthermore, $H^+(x)\ll T^{-A}$ unless $x\gg \Delta T^{1-\varepsilon}$,
  in which case we have  $H^+(x)\ll T\Delta x^{-1/2}$.
\end{lemma}

\begin{proof}
  This is Young~\cite[Lemma 7.1]{young2014weyl}.
  See also Huang \cite[Lemma 4.2]{Huang2021}.
\end{proof}

\begin{lemma}\label{lemma:H-}
  Let $H^-$ be given by \eqref{eqn: H} with $h(t)= h_{T,\Delta}(t)$.
  There exists a function $g$ depending on $T$ and $\Delta$ satisfying $g^{(j)}(y)\ll_{j,A} (1+|y|)^{-A}$,
  so that
  \begin{equation}\label{eqn: H^-=}
    H^-(x) = \Delta T\int_{|v|\leq \frac{\Delta^\varepsilon}{\Delta}} \cos(x\sinh(v))e\left(\frac{vT}{\pi}\right)g(\Delta v) \dd v + O(T^{-A}).
  \end{equation}
  Furthermore, $H^-(x)\ll (x+T)^{-A}$ unless $x\asymp T$,
   in which case we have $H^{-}(x)\ll T^{1+\varepsilon}$.
\end{lemma}

\begin{proof}
  This is  Young~\cite[Lemma 7.2]{young2014weyl}.
  See also Huang \cite[Lemma 4.3]{Huang2021}.
\end{proof}




\subsection{Large sieve inequalities}
\begin{lemma}\label{lemma:Largesieve}
Suppose $U\geq 1$, and let $\{a_n\}_{n\geq 1}$ be a complex sequence. Then
    \[
        \int_{|u|\leq U}\left|\sum_{n\leq N}a_{n}n^{iu}\right|^2\dd u\ll (U+N)\sum_{n\leq N}|a_n|^2.
        \]
\end{lemma}
\begin{proof}
See \cite{Gallager1970}.
\end{proof}
\begin{lemma}\label{lemma:Largesieve2}
Suppose $U\geq 1$, and let $\{a_{m,n}\}_{m,n\geq 1}$ be a two dimensional complex array which is supported on co-prime lattice, i.e. $a_{m,n}=0$ for $(m,n)\neq 1$. Then
    \[
        \int_{|u|\leq U}\left|\sum_{n\leq N}\sum_{m\leq M}a_{m,n}\left(\frac{m}{n}\right)^{iu}\right|^2\dd u\ll (MN)^\varepsilon(U+MN)\sum_{m\leq M}\sum_{n\leq N}|a_{m,n}|^2,
    \]
where the implied constant only depends on $\varepsilon$.
\end{lemma}
\begin{proof}Let $Y\geq 1$ and consider the weight function $f$ satisfies that $f$ supports on the interval $(-Y-U, U+Y)$, $f(u)\equiv 1$ on the segment $|u|\leq U$ and $f^{(j)}(u)\ll_j \frac{1}{Y^j}$ for any positive integer $j$.
The left hand side is less than
\[
    \int_{-\infty}^{\infty}f(u)\left|\sum_{n\leq N}\sum_{m\leq M}a_{m,n}\left(\frac{m}{n}\right)^{iu}\right|^2\dd u\\
    =\sum_{m_1\leq M}\sum_{n_1\leq N}\sum_{m_2\leq M}\sum_{n_2\leq N}a_{m_1, n_1}\bar{a}_{m_2,n_2}F\left(\frac{m_1n_2}{m_2n_1}\right)
\]
where $F(x)=\int_{-\infty}^{\infty}f(u)x^{iu}\dd u$.
Note that
\[
    F(1)=2U+O(Y),
\]
and by integration by parts,
\[
F(x)\ll_j \frac{Y}{(Y|\log x|)^j}\ll \frac{1}{Y^{j-1}}\left|\frac{1+x}{1-x}\right|^j,
\] if $x\neq 1$. Here we use the fact that $|\frac{1-x}{1+x}|\ll |\log x|.$ Indeed, if $x$ is close to 1, it can be obtained by taking Taylor expansion of $\log x$, otherwise,
$|\frac{1-x}{1+x}|$ is bounded.\par
For the diagonal term, that is $m_1n_2=m_2n_1$.
Due to $(m_1,n_1)=(m_2,n_2)=1$, we get $m_1=m_2, n_1=n_2$.
Therefore, the diagonal term
\[
    \ll (U+Y)\sum_{m\leq M}\sum_{n\leq N}|a_{m,n}|^2.
\]
The off-diagonal term contributes
 \begin{equation*}
    \begin{split}
        &  \mathop{\sum\sum\sum\sum}_{n_1, n_2, m_1, m_2\atop m_1n_2\neq m_2n_1 }|a_{m_1, n_1}{a_{m_2,n_2}}|\frac{1}{Y^{j-1}}\left|\frac{m_1n_2+m_2n_1}{m_1n_2-m_2n_1}\right|^j\\
        &  \hskip 20pt\leq 2\frac{(MN)^j}{Y^{j-1}}\sum_{m_1\leq M}\sum_{n_1\leq N}|a_{m_1, n_1}|^2
        \mathop{\sum\sum}_{n_2\leq N, m_2\leq M\atop m_1n_2\neq m_2n_1 }\frac{1}{|m_1n_2-m_2n_1|^j}
    \end{split}
  \end{equation*}
By using inequality $ab \leq a^2+b^2$ ,
and the fact that fixing $m_1, n_1$ and integer $k\geq 1$,
the number of solution $(m_2, n_2)$ satisfying the equation $m_1n_2-m_2n_1=k$ is less than $ \min\{M, N\}\leq 2(MN)^\frac{1}{2}$,
let $j\geq 2$, the above sum is bounded by
\begin{equation*}
    2\frac{(MN)^j}{Y^{j-1}}\sum_{m_1\leq M}\sum_{n_1\leq N}|a_{m_1, n_1}|^2\sum_{k\geq 1}\frac{2(MN)^\frac{1}{2}}{k^j}
    \ll\sum_{m\leq M}\sum_{n\leq N}|a_{m, n}|^2\frac{(MN)^{j+\frac{1}{2}}}{Y^{j-1}}.
\end{equation*}
Combining the diagonal contribution and the off-diagonal contribution and taking $Y=(MN)^{1+\varepsilon},$ $j=2+\left[\frac{1}{2\varepsilon}\right]$, this completes the proof of the lemma.
\end{proof}

\subsection{Oscillatory integrals}
Let $\mathcal{F}$ be an index set and $X=X_T: \mathcal{F} \rightarrow \mathbb{R}_{\geq 1}$ be a function of $T \in \mathcal{F}$.
 A family $\left\{w_T\right\}_{T \in \mathcal{F}}$ of smooth functions supported on a product of dyadic intervals in $\mathbb{R}_{>0}^d$ is called $X$-inert if for each $j=\left(j_1, \cdots, j_d\right) \in \mathbb{Z}_{\geq 0}^d$ we have
\begin{equation}\label{eqn: def_inert}
    C_{\mathcal{F}}\left(j_1, \cdots, j_d\right):=\sup _{T \in \mathcal{F}} \sup _{\left(x_1, \cdots, x_d\right) \in \mathbb{R}_{>0}^d} X_T^{-j_1-\cdots-j_d}\left|x_1^{j_1} \cdots x_d^{j_d} w_T^{\left(j_1, \cdots, j_d\right)}\left(x_1, \ldots, x_d\right)\right|<\infty .
\end{equation}
We will use the following lemmas several times.
\begin{lemma}\label{lemma: spl_1}
 Suppose that $w=w_T(t)$ is a family of $X$ inert functions, with compact support on $[Z, 2 Z]$, so that for all $j=0,1, \ldots$ we have the bound $w^{(j)}(t) \ll(Z / X)^{-j}$. Also suppose that $\phi$ is smooth and satisfies, for $j=2,3, \cdots$, $\phi^{(j)}(t) \ll \frac{Y}{Z^j}$ for some $R \geq 1$ with $Y / X \geq R$ and all $t$ in the support of $w$. Let
    \begin{equation}
    I=\int_{-\infty}^{\infty} w(t) e^{i \phi(t)} \dd t.
    \end{equation}
    If $\left|\phi^{\prime}(t)\right| \gg \frac{Y}{Z}$ for all $t$ in the support of $w$, then $I \ll_A Z R^{-A}$ for $A$ arbitrarily large.
\end{lemma}
\begin{proof}
    See \cite{BKY13} and \cite{KPY19}.
\end{proof}
\begin{lemma}\label{lemma: spl_2}
  Suppose $w_T$ is $X$-inert in $t_1, \cdots, t_d$, supported on $t_1 \asymp Z$ and $t_i \asymp X_i$ for $i=2, \cdots, d$. Suppose that on the support of $w_T, \phi=\phi_T$ satisfies
\[
\frac{\partial^{a_1+a_2+\cdots+a_d}}{\partial t_1^{a_1} \cdots \partial t_d^{a_d}} \phi\left(t_1, t_2, \cdots, t_d\right) \ll_{C_{\mathcal{F}}} \frac{Y}{Z^{a_1}} \frac{1}{X_2^{a_2} \cdots X_d^{a_d}},
\]
for all $a_1, \cdots, a_d \in \mathbb{N}$ with $a_1 \geq 1$. Suppose $\phi^{\prime \prime}\left(t_1, t_2, \ldots, t_d\right) \gg \frac{Y}{Z^2}$ (here and later, $\phi^{\prime}$ and $\phi^{\prime \prime}$ denote the derivative with respect to $\left.t_1\right)$, for all $t_1, t_2, \cdots, t_d$ in the support of $w_T$, and for each $t_2, \cdots, t_d$ in the support of $\phi$ there exists $t_0 \asymp Z$ such that $\phi^{\prime}\left(t_0, t_2, \ldots, t_d\right)=0$. Suppose that $Y / X^2 \geq R$ for some $R \geq 1$. Then
\begin{equation}\label{eqn: spm}
I=\int_{\mathbb{R}} e^{i \phi\left(t_1, \cdots, t_d\right)} w_T\left(t_1, \cdots, t_d\right) d t_1=\frac{Z}{\sqrt{Y}} e^{i \phi\left(t_0, t_2, \cdots, t_d\right)} W_T\left(t_2, \cdots, t_d\right)+O_A\left(Z R^{-A}\right),
\end{equation}
for some $X$-inert family of functions $W_T$, and where $A>0$ may be taken to be arbitrarily large. The implied constant in equation \eqref{eqn: spm} depends only on $A$ and on $C_{\mathcal{F}}$ defined in formula \eqref{eqn: def_inert}.
\begin{proof}
    See \cite{BKY13} and \cite{KPY19}.
\end{proof}
\end{lemma}

\section{Applying the Kuznetsov trace formula }\label{sec:appyingKTT}
Applying the approximate functional equations (Lemmas \ref{lemma:AFE2} and \ref{lemma:AFE3}), we have
\begin{equation}
 \begin{split}
  \mathcal M &:= \sideset{}{'}\sum_{j\geq1}
  w_j h_{T,\Delta}(t_j)
   L(1/2,u_j) L(1/2,\sym^2 u_j) \\
   &\hskip 100pt +\frac{1}{4\pi}\int_{-\infty}^{\infty}h_{T,\Delta}(t)w(t)\zeta(1/2)|\zeta(1/2+it)\zeta(1/2+2it)|^2\dd t\\
   & =
  4\sum_{m\geq1} \frac{1}{m^{1/2}}
  \sum_{n\geq1} \frac{1}{n^{1/2}}
  \left(\sideset{}{'}\sum_{j\geq1}
  w_j h_{T,\Delta}(t_j)
  W\left(m,t_j\right) V\left(n,t_j\right)
  \lambda_j(m) \lambda_j(n^2)\right.\\
  &\hskip 110pt \left.+\frac{1}{4\pi}\int_{-\infty}^{\infty}h_{T,\Delta}(t)w(t)W\left(m,t\right) V\left(n,t\right)
  \eta_t(m) \eta_t(n^2)\dd t\right).
  \end{split}
\end{equation}

Applying the Kuznetsov trace formula (Lemma \ref{lemma:KTF}), we get
\begin{equation}
  \mathcal {M} = \mathcal{ D } + \sum_{\pm}\mathcal{R}^\pm,
\end{equation}
where
\begin{equation}\label{eqn:D}
  \mathcal{ D }:= \frac{4}{\pi}
  \sum_{n\geq1} \frac{1}{n^{\frac{3}{2}}} \int_{0}^{\infty} h_{T,\Delta}(t)
  W\left(n^2,t\right) V\left(n,t\right) \tanh(\pi t)t \dd t,
\end{equation}
and
\begin{equation}\label{eqn:R}
    \mathcal{R}^\pm :=
    2 \sum_{m\geq1} \frac{1}{m^{1/2}}
  \sum_{n\geq1} \frac{1}{n^{1/2}} \sum_{c\geq1}\frac{S(n^2,\pm m; c)}{c} \widetilde{H}^{\pm}\left(\frac{4\pi\sqrt{m}n}{c}\right),
\end{equation}
\begin{equation}\label{eqn:widetilde{H}+=}
       \widetilde{H}^+(x) = 2i \int_{-\infty}^{\infty} J_{2it}(x)\frac{h_{T,\Delta}(t)
   W\left(m,t\right) V\left(n,t\right)t}{\cosh(\pi t)} \dd t,
\end{equation}
\begin{equation}\label{eqn:widetilde{H}-=}
       \widetilde{H}^-(x) = \frac{4}{\pi} \int_{-\infty}^{\infty} K_{2it}(x)\sinh(\pi t) h_{T,\Delta}(t)
   W\left(m,t\right) V\left(n,t\right)t \dd t.
\end{equation}
By the Cauchy--Schwarz inequality,  the contribution of continuous spectrum is bounded by
\begin{equation*}
    T^\varepsilon\left(\int_{t\asymp T}|\zeta(1/2+it)|^4\dd t\right)^{\frac{1}{2}}
    \left(\int_{t\asymp T}|\zeta(1/2+2it)|^4\dd t\right)^{\frac{1}{2}}
    \ll T^{1+\varepsilon}
\end{equation*}
Therefore, we have
\begin{equation}\label{eqn: moment=D+R}
     \sideset{}{'}\sum_{j\geq1}
     w_j h_{T,\Delta}(t_j)
     (1/2,u_j) L(1/2,\sym^2 u_j)= \mathcal{ D } + \sum_{\pm}\mathcal{R}^\pm+O(T^{1+\varepsilon})
\end{equation}

\section{The diagonal term}\label{sec:diagonal}
By Lemmas \ref{lemma:AFE2} and \ref{lemma:AFE3}, we write
\begin{equation}\label{eqn:D=int}
  \mathcal{ D } = \frac{4}{\pi} \int_{0}^{\infty} h_{T,\Delta}(t)
  \mathcal{P}(t) \tanh(\pi t)t \dd t,
\end{equation}
where
\begin{multline}\label{eqn:P=int}
      \mathcal{P}(t) =\frac{1}{(2\pi i)^2} \int_{(2)}\int_{(2)}\frac{\prod_{\pm}\Gamma_\mathbb{R}(1/2+s_1\pm it)}{\prod_{\pm}\Gamma_\mathbb{R}(1/2\pm it)^2}\frac{\Gamma_\mathbb{R}(1/2+s_2)}{\Gamma_\mathbb{R}(1/2)}\frac{\prod_{\pm}\Gamma_\mathbb{R}(1/2+s_2\pm 2it)}{\prod_{\pm}\Gamma_\mathbb{R}(1/2\pm 2it)^2}\\
      \times\zeta(1+2s_2)\zeta(3/2+2s_1+s_2)G(s_1)G(s_2) \frac{\dd s_1}{s_1}\frac{\dd s_2}{s_2}.
\end{multline}
Here the zeta function $\zeta(3/2+2s_1+s_2)$ comes from the $n$-sum in \eqref{eqn:D}. By Stirling's formula, we truncate the line integral and get
\begin{multline}\label{eqn:P=intasym}
      \mathcal{P}(t) =\frac{1}{(2\pi i)^2} \int_{\varepsilon-i(\log T)^2}^{\varepsilon+i(\log T)^2}\int_{\varepsilon-i(\log T)^2}^{\varepsilon+i(\log T)^2}\frac{1}{2^{s_1}}\left(\frac{t}{\pi}\right)^{s_1+s_2}\frac{\Gamma_\mathbb{R}(1/2+s_2)}{\Gamma_\mathbb{R}(1/2)}
    \\
    \times \zeta(1+2s_2)\zeta(3/2+2s_1+s_2)G(s_1)G(s_2) \frac{\dd s_1}{s_1}\frac{\dd s_2}{s_2}+O(T^{-2+\varepsilon}).
\end{multline}
Due to the weight function $h(t)$,
it suffices to consider the asymptotic behaviour of $\mathcal{P}(t)$ for $t\asymp T$.
We shift the $s_1$-integral to  $\Re{s_1}=-1/4$, crossing a simple pole at $s_1=0$, the shifted integral and the horizonal integral has size of $O(T^{-\frac{1}{4}+\varepsilon})$.
The residue is
\[
    \frac{1}{2\pi i} \int_{\varepsilon-i(\log T)^2}^{\varepsilon+i(\log T)^2}\left(\frac{t}{\pi}\right)^{s_2}\frac{\Gamma_\mathbb{R}(1/2+s_2)}{\Gamma_\mathbb{R}(1/2)}\zeta(1+2s_2)\zeta(3/2+s_2)G(s_2)\frac{\dd s_2}{s_2}
\]
We move the line of $s_2$-integral to  $\Re{s_2}=-1/4$, crossing the pole at $s_2=0$ of order $2$.
The shifted integral and the horizonal integral are bounded by $O(T^{-\frac{1}{4}+\varepsilon})$.
Using expansion $\zeta(1+z)=\frac{1}{z}+\gamma+c_1z+\cdots$, from the residue,
we get
\begin{equation}\label{eqn:P=logt}
    \mathcal{P}(t)=\frac{1}{2}\zeta\left(\frac{3}{2}\right)\log t+\left(\zeta(\frac{3}{2})\left(-\frac{3}{4}\log \pi +\gamma+\frac{1}{4}\psi(\frac{1}{4})\right)+\frac{1}{2}\zeta'(\frac{3}{2})\right)+O(T^{-\frac{1}{4}+\varepsilon}),
\end{equation}
where $\gamma$ is the Euler constant.
Finally, evaluating the $t$-integral through (\ref{eqn:P=logt}), using $\tanh(\pi t)=1+O(e^{-2\pi |t|})$, we obtain
\begin{multline}\label{eqn:Diagonal_asymp}
      \mathcal{D} =\frac{2}{\pi}\zeta\left(\frac{3}{2}\right)H_{T,\Delta}^{\log}\\+
      \frac{4}{\pi}\left(\zeta(\frac{3}{2})\left(\frac{3\gamma}{4}-\frac{3}{4}\log \pi -\frac{3\log 2}{4}-\frac{\pi}{8}\right)+\frac{1}{2}\zeta'(\frac{3}{2})\right)H_{T,\Delta}+O(\Delta T^{\frac{3}{4}+\varepsilon}),
\end{multline}
where
  \[
        H_{T,\Delta}^{\log}=\int_{0}^{\infty}e^{-\frac{(t-T)^2}{\Delta^2}}t\log t\,\dd t,
  \]
and
    \[
     H_{T,\Delta}=\int_{0}^{\infty}e^{-\frac{(t-T)^2}{\Delta^2}}t\,\dd t,
    \]
\section{The off-diagonal terms}\label{sec:offdiagonal}
Using (\ref{eqn:Wt=WT}) and (\ref{eqn:Vt=VT}), we get that the off-diagonal term becomes
\begin{equation}\label{eqn: R=}
   \mathcal{R}^\pm=\sum_{0\leq l_1, l_2< L}\mathcal{R}_{l_1, l_2}^\pm+O\left((\Delta T^{-1})^{2L}T^{O(1)}\right),
\end{equation}
where
\begin{equation}\label{eqn:R_{l_1, l_2}}
    \mathcal{R}_{l_1, l_2}^\pm :=
    2 \sum_{m\geq1} \frac{W(m, T)}{m^{1/2}}
  \sum_{n\geq1} \frac{V(n, T)}{n^{1/2}} \sum_{c\geq1}\frac{S(n^2,\pm m; c)}{c} H_{l_1, l_2}^{\pm}\left(\frac{4\pi\sqrt{m}n}{c}\right),
\end{equation}
\begin{equation}\label{eqn:H_{l_1, l_2}+=}
       H_{l_1, l_2}^+(x) = 2i \int_{-\infty}^{\infty} J_{2it}(x)\frac{h_{T,\Delta}(t)
       \left(\frac{t^2-T^2}{T^2}\right)^{l_1+l_2}}{\cosh(\pi t)} \dd t,
\end{equation}
\begin{equation}\label{eqn:H_{l_1, l_2}-=}
       H_{l_1, l_2}^-(x) = \frac{4}{\pi} \int_{-\infty}^{\infty} K_{2it}(x)\sinh(\pi t) h_{T,\Delta}(t)
       \left(\frac{t^2-T^2}{T^2}\right)^{l_1+l_2}t \dd t.
\end{equation}
Here $L$ is a large integer. The error term in \eqref{eqn: R=} is depending on $L$ with negative power of $T$. For example, for $\Delta\leq T^{1-\varepsilon}$,
we can take $L=[\varepsilon^{-2}]$, the error term becomes $O(T^{-\frac{2}{\varepsilon}+O(1)})$ which is negligible.

We only consider the case $l_1=l_2=0$ (simply denote $H_{0, 0}^\pm$ by $H^\pm$), since we can save additional power of $T$ in the other cases.
Taking a smooth dyadic partition of unity to the $m,n$-sum, it suffices to bound the sum
\begin{equation}
    \mathcal{S}_{\pm}(M, N; C)=\sum_{m\geq1}
  \sum_{n\geq1} w_M\left(\frac{m}{M}\right) w_N\left(\frac{n}{N}\right) \sum_{c\asymp C}\frac{S(n^2,\pm m;c)}{c} H^{\pm}\left(\frac{4\pi\sqrt{m}n}{c}\right),
\end{equation}
where $M, N\ll T^{1+\varepsilon}$, $C\geq 1$ and
 $w_M(x),w_N(x)$ are smooth weight functions supported on $(\frac{1}{2},3)$, satisfying
 $x^jw_M^{(j)}(x)\ll_j 1$ and $x^jw_N^{(j)}(x)\ll_j \log T$ respectively.

According to the weak bound $H^{\pm}(x)\leq Tx^{3/4}$ and the Weil bound for the Kloosterman sum, we obtain the trivial bound
\begin{equation}
     \mathcal{S}_{\pm}(M, N; C)\ll M^{15/8}N^{7/4}TC^{-1/4+\varepsilon}.
\end{equation}
Thus, we can restrict $C$ to be less than $T^{B}$ for some fixed large $B$, and
\begin{equation}\label{eqn: R_bound}
    \mathcal{R}^{\pm}\ll T^{1+\varepsilon}\sup_{M,N\ll T^\varepsilon, C\ll T^B}\frac{|\mathcal{S}_{\pm}(M, N; C)|}{\sqrt{MN}}+O(T^{-10}).
\end{equation}
From \eqref{eqn: moment=D+R}, \eqref{eqn:Diagonal_asymp} and \eqref{eqn: R_bound}, one can see that Theorem \ref{thm:moment} is reduced to prove the following proposition.
\begin{proposition}\label{proposition:S}
Assume that $C\ll T^B$ for some fixed large $B$ and $M, N\ll T^{1+\varepsilon}$, Then
    \[
         \mathcal{S}_{\pm}(M, N; C)\ll (MN)^\frac{1}{2}T^{1+\varepsilon}.
    \]
\end{proposition}
\begin{remark}\label{remark:Pro. holds for H_0}
Define
\begin{equation}\label{eqn:H_0}
    H_0(x)=\Delta T\int_{|v|\leq \frac{\Delta^\varepsilon}{\Delta}} e\left(\frac{x\phi(v)}{2\pi}+ \frac{vT}{\pi}\right)g(\Delta v) \dd v .
\end{equation}
where
\begin{equation}\label{eqn:function_phi_def}
    \phi(v)\in \{\pm \cosh(v), \pm \sinh(v), \pm \cos(v), \pm\sin(v)\}.
\end{equation}
and $g$ is a function satisfying $g^{j}(x)\ll_{j,A}(1+|x|)^{-A}$. Then we can show that Proposition \ref{proposition:S} holds for replacing $H^{\pm}$ by $H_0$. See the discussion of \cite[Proposition 7.3]{young2014weyl} and the proof of Lemma \ref{lemma:I_phi}.
\end{remark}
\par
Applying the Poisson summation formula to $m, n$ modulo $c$, we get
\begin{equation}\label{eqn:S=}
     \mathcal{S}_{\pm}(M, N; C)=\sum_{c\asymp C}\mathop{\sum\sum}_{x_1, x_2\in \mathbb{Z}}W_{\pm}(x_1, x_2, c)G_{\pm}(x_1, x_2, c),
\end{equation}
where
\begin{equation}\label{eqn:W}
     W_{\pm}(x_1, x_2; c )=\int_{0}^{\infty}\int_{0}^{\infty}H^\pm\left( \frac{4\pi \sqrt{t_1}t_2}{c}\right) w_M\left(\frac{t_1}{M}\right) w_N\left(\frac{t_2}{N}\right)e\left(-\frac{x_1t_1}{c}-\frac{x_2t_2}{c}\right) \dd t_1 \dd t_2
\end{equation}
and
\begin{equation}\label{eqn:G=}
     G_\pm(x_1, x_2, c)=\frac{1}{c^3}\sum_{a, b\bmod c}S(\pm a, b^2; c)e\left(\frac{ax_1}{c}+\frac{bx_2}{c}\right).
\end{equation}

We separate the sum (\ref{eqn:S=}) into two parts, say
\begin{equation}\label{eqn:S=S_1+S_2}
     \mathcal{S}_{\pm}(M, N; C)= \mathcal{S}_{\pm}^{'}(M, N, C) +\mathcal{S}_{\pm}^{''}(M, N; C),
\end{equation}
where
\begin{equation}
\mathcal{S}_{\pm}^{'}(M, N; C)=\sum_{c\asymp C}\mathop{\sum\sum}_{x_1, x_2\in \mathbb{Z}\atop |x_1|M/C\geq T^\varepsilon \text{or} |x_2|N/C\geq T^\varepsilon}W_{\pm}(x_1, x_2, c)G_{\pm}(x_1, x_2, c),
\end{equation}
and
\begin{equation}
    \mathcal{S}_{\pm}^{''}(M, N; C)=\sum_{c\asymp C}\mathop{\sum\sum}_{x_1, x_2\in \mathbb{Z}\atop |x_1|M/C< T^\varepsilon, |x_2|N/C< T^\varepsilon}W_{\pm}(x_1, x_2, c)G_{\pm}(x_1, x_2, c).
\end{equation}

\subsection{Analysis of oscillatory integral}
By conjugation of the integral, we have
\begin{equation}\label{eqn:W=barW}
    W_{\pm}(-x_1, -x_2; c )=\overline{W_{\pm}(x_1, x_2; c )}.
\end{equation}
 Change variables $u=t_1\sqrt{t_2}$, $t=t_2$, we have
    \[
       W_{\pm}(x_1, x_2; c )=\int_{0}^{\infty}H^\pm\left( \frac{4\pi \sqrt{u}}{c}\right)\Upsilon_{x_1, x_2, c}(u) \dd u.
    \]
where
\begin{equation}\label{eqn: Upsilon_def}
\Upsilon_{x_1, x_2; c}(u)=\int_{0}^{\infty}w_M\left(\frac{u^2}{Mt}\right) w_N\left(\frac{t}{N}\right)e\left(-\frac{x_1u}{ct^2}-\frac{x_2t}{c}\right) \frac{1}{t^2}\dd t
\end{equation}
The weight functions force $t\asymp N$ and $u\asymp MN^2$.
This $t$-integral can be analyzed by the stationary phase method, and we have the following lemma.
\begin{lemma}\label{lemma:Upsilon}
Let $\Upsilon_{x_1, x_2; c}(u)$ be given by (\ref{eqn: Upsilon_def}), $x_1\asymp X_1$, $x_2\asymp X_2$ and $c\asymp C\geq 1$.
Assume that
  \begin{equation}\label{assump:X_1M/C X_2N/C not small}
        |X_1|M/C\gg T^{\varepsilon} \quad \text{or } \quad |X_2|N/C\gg T^{\varepsilon}
  \end{equation}
Then we have
    \[
        \Upsilon_{x_1, x_2; c}(u)\ll T^{-A},
    \]
unless
    \begin{equation}\label{assump:X_1M=X_2N}
         X_1M\asymp X_2N  \quad \text{and } \quad x_1x_2>0,
    \end{equation}
holds, in which case
\begin{equation}\label{eqn:Upsilon_approxi}
    \Upsilon_{x_1, x_2; c}(u)=\frac{c^{\frac{1}{2}}{x_1}^{\frac{1}{6}}M^{\frac{1}{6}}}{{x_2}^{\frac{2}{3}}N^{\frac{5}{3}}}e\left(-\frac{3\sqrt[3]{2}}{2}\frac{{x_1}^{\frac{1}{3}}{x_2}^{\frac{2}{3}}{u}^{\frac{1}{3}}}{c}\right)w_1(u)+O(T^{-A}),
\end{equation}
for some $1$-inert function $w_1(u)$ which is supported on $u\asymp MN^2$.
\end{lemma}
\begin{proof}
    The phase function of the oscillatory integral is
    \[
        h_1(t)=-\frac{x_1u}{ct^2}-\frac{x_2t}{c},
    \]
    we have
    \[
        h_{1}^{'}(t)=\frac{2x_1u}{ct^3}-\frac{x_2}{c}
    \]
    and
    \[
        h_{1}^{''}(t)=-\frac{6x_1u}{ct^4}, \quad h_{1}^{j}(t)\asymp_j \frac{X_1M}{CN^{j}}, j\geq 2.
    \]
    If $X_1M \ggg X_2N$ or $X_1M \lll X_2N$, and $sgn(x_1)\neq sgn(x_2)$, then we have $|h_{1}^{'}(t)|\gg (|X_1|M+|X_2|N)/C\gg T^{\varepsilon}$.
    By Lemma \ref{lemma: spl_1}, we have
    \[
        \Upsilon_{x_1, x_2; c}(u)\ll T^{-A}
    \]
    for any $A>0$.
    Now suppose $X_1M \asymp X_2N$, and $sgn(x_1)\neq sgn(x_2)$.
    The stationary point is
        $t_0=\left(\frac{2x_1u}{x_2}\right)^\frac{1}{3}$
    i.e.  $h_{1}^{'}(t_0)=0$. By Lemma \ref{lemma: spl_2}, $\Upsilon_{x_1, x_2; c}(u)$ is equal to
\[
    \frac{c^{\frac{1}{2}}{x_1}^{\frac{1}{6}}M^{\frac{1}{6}}}{{x_2}^{\frac{2}{3}}N^{\frac{5}{3}}}e\left(-\frac{3\sqrt[3]{2}}{2}\frac{{x_1}^{\frac{1}{3}}{x_2}^{\frac{2}{3}}{u}^{\frac{1}{3}}}{c}\right)w_1(u),
     +O(T^{-A})
\]
for some $1$-inert function $w_1(u)$ which is supported on $u\asymp MN^2$.
\end{proof}
For  $x_1\asymp X_1$, $x_2\asymp X_2$ and $c\asymp C$, we always assume that (\ref{assump:X_1M/C X_2N/C not small}) holds in this subsection. The remaining case will be discussed in the subsection \ref{subsection: remaining-case}.\par

If (\ref{assump:X_1M=X_2N}) holds, we have
\begin{equation}\label{eqn:WunderX_1M=X_2N}
    \begin{split}
    W_{\pm}(x_1, x_2; c )=\frac{c^{\frac{1}{2}}{x_1}^{\frac{1}{6}}M^{\frac{1}{6}}N^{\frac{1}{3}}}{{x_2}^{\frac{2}{3}}}&\int_{u\asymp MN^2}H^\pm\left( \frac{4\pi \sqrt{u}}{c}\right)\\
    &\times e\left(-\frac{3\sqrt[3]{2}}{2}\frac{{x_1}^{\frac{1}{3}}{x_2}^{\frac{2}{3}}{u}^{\frac{1}{3}}}{c}\right)w_1(u) \dd u+O(T^{-A}).
    \end{split}
\end{equation}
In this case, by (\ref{eqn:W=barW}), we can assume that $X_1, X_2\geq 1$.
Otherwise, by Lemma \ref{lemma:Upsilon}, $W_{\pm}(x_1, x_2; c )$ is small.
\par
We define
\begin{equation}\label{eqn:Iphi_def}
    I_{\phi}(x_1, x_2; c):=\frac{c^{\frac{1}{2}}{x_1}^{\frac{1}{6}}M^{\frac{1}{6}}\Delta T}{{x_2}^{\frac{2}{3}}N^{\frac{5}{3}}}\int_{|v|\leq \Delta^{\varepsilon-1}}e\left(\frac{vT}{\pi}\right)g(\Delta v)\Lambda_{\phi}(v; x_1, x_2, c) \dd v,
\end{equation}
where
\begin{equation}\label{eqn:Lambdaphi_def}
    \Lambda_{\phi}(v; x_1,x_2,c)
    =\int_{0}^{\infty}w_1(u)e\left(\frac{2\sqrt{u}\phi(v)}{c}-\frac{3\sqrt[3]{2}}{2}
    \frac{x_1^{\frac{1}{3}}x_2^{\frac{2}{3}}}{c}{u}^{\frac{1}{3}}\right)\dd u,
\end{equation}
and
\begin{equation}\label{eqn:function_phi_def}
    \phi(v)\in \{\pm \cosh(v), \pm \sinh(v), \pm \cos(v), \pm\sin(v)\}.
\end{equation}
By Lemmas \ref{lemma:H+} and \ref{lemma:H-}, we have,
if $\frac{\sqrt{M}N}{C}\gg T^{1-\varepsilon}\Delta$,
\begin{equation*}
    \begin{split}
    W_{+}(x_1, x_2; c )=I_{\cosh}(x_1, x_2; c)+I_{-\cosh}(x_1, x_2; c)+O(T^{-100}),
    \end{split}
\end{equation*}
and
if $\frac{\sqrt{M}N}{C}\asymp T $,
\begin{equation*}
    \begin{split}
    W_{-}(x_1, x_2; c )=I_{\sinh}(x_1, x_2; c)+I_{-\sinh}(x_1, x_2; c)+O(T^{-100}),
    \end{split}
\end{equation*}

From the above discussion, the analytic part $W_{\pm}$ is essentially the integral $I_{\phi}$.
The following lemma helps us to analyse the behavior of $I_{\phi}(x_1, x_2; c)$.
\begin{lemma}\label{lemma:I_phi}
Let $I_{\phi}(x_1, x_2; c)$ be given by (\ref{eqn:Iphi_def}), (\ref{eqn:Lambdaphi_def}) and (\ref{eqn:function_phi_def}) with $x_1\asymp X_1\geq 1$, $x_2\asymp X_2\geq 1$ and $c\asymp C\geq 1$. Set $X=\frac{X_1{X_2}^2}{C}$, We have
\[
    I_{\phi}(x_1, x_2; c)\ll T^{-A},
\]
unless \begin{equation}\label{condition:X_1X_2^2}
    \left\{
    \begin{aligned}
    &{X_1}{X_2}^2\asymp M^{\frac{1}{2}}N, &    &\phi(v)=\pm \cosh(v) \text{ or } \pm \cos(v),\\
    &{X_1}{X_2}^2\ll M^{\frac{1}{2}}N\Delta^{\varepsilon-3},&    &\phi(v)=\pm \sinh(v) \text{ or } \pm \sin(v).
    \end{aligned}
    \right.
\end{equation}
If (\ref{condition:X_1X_2^2}) holds, then we have
\begin{equation}\label{eqn:Iphi=eint2}
    \begin{split}
    I_{\phi}(x_1, x_2; c)=\frac{cM\Delta}{x_2}e\left(\mp\frac{x_1{x_2}^2}{4c}\right)\int_{|u|\ll U}\lambda_{X, T}(u)\left(\frac{x_1{x_2}^2}{4c}\right)^{iu}\dd u+O(T^{-A}).
    \end{split}
\end{equation}
where $\lambda_{X,T}$ and $U$ depend on $X$, $T$, and the choice of $\phi(v)$, $\lambda_{X,T}$ satisfies $\lambda_{X,T}(u)\ll 1$, and
\begin{equation}\label{eqn:U}
    \left\{
    \begin{aligned}
    &U=T^2/X,&   &\phi(v)=\pm \cosh(v) \text{ or } \pm \cos(v),\\
    &U=X^{\frac{1}{3}}T^\frac{2}{3},&   &\phi(v)=\pm \sinh(v) \text{ or } \pm \sin(v).
    \end{aligned}
    \right.
\end{equation}
the $\mp$ sign in (\ref{eqn:Iphi=eint2}) is depend on $\phi$, it is $-$ for $\pm \cosh(v)$ or $\pm \cos(v)$ cases, and
$+$ for $\pm \sinh(v)$ or $\pm \sin(v)$ cases.
\end{lemma}
\begin{proof}
Suppose that the integral $\Lambda_{\phi}(v; x_1, x_2, c)$ exists a stationary point, say $u_0=\frac{1}{16}\frac{{x_1}^2{x_2}^4}{{\phi(v)}^6}$, which implies
\begin{equation}\label{eqn:phiv}
    |\phi(v)|\asymp {X_1}^{\frac{1}{3}}{X_2}^{\frac{2}{3}}M^{-\frac{1}{6}}N^{-\frac{1}{3}},
\end{equation}
otherwise it is small.\par
Since $|v|$ is small,
if $\phi(v)=\pm\cosh(v)$, then (\ref{eqn:phiv}) shows ${X_1}{X_2}^2\asymp M^{\frac{1}{2}}N$, which leads to the condition $x_1\asymp \frac{N}{M^\frac{1}{2}},  x_2\asymp M^{\frac{1}{2}}$ by (\ref{assump:X_1M=X_2N}).
If $\phi(v)=\pm\sinh(v)$, then $|v|\asymp {X_1}^{\frac{1}{3}}{X_2}^{\frac{2}{3}}M^{-\frac{1}{6}}N^{-\frac{1}{3}}$, note that $|v|\leq \Delta^{\varepsilon-1}$, which leads to the condition ${X_1}{X_2}^2\ll M^{\frac{1}{2}}N\Delta^{\varepsilon-3}$.
Combining with (\ref{assump:X_1M=X_2N}),
it implies that $X_1\ll M^{-\frac{1}{2}}N\Delta^{\varepsilon-1}, X_2\ll M^{\frac{1}{2}}\Delta^{\varepsilon-1}$.
By Lemma \ref{lemma: spl_2}, we have
\begin{equation}\label{eqn:Lambda=int}
    \Lambda_{\phi}(v; x_1, x_2, c)
    =\frac{c^\frac{1}{2}M^\frac{5}{6}N^\frac{5}{3}}{{x_1}^\frac{1}{6}{x_2}^\frac{1}{3}}e\left(-\frac{x_1{x_2}^2}{4c\phi(v)^2}\right)w_{2, \phi}(v)+O(T^{-A}),
\end{equation}
where $w_{2, \phi}(v)$ is an inert function supported on $|v|\ll \Delta^{\varepsilon-1}$.
Put the above asymptotic into $I_{\phi}(x_1, x_2; c)$, obtain
\begin{equation}\label{eqn:Iphi=eint}
    \begin{split}
    I_{\phi}(x_1, x_2; c)=\frac{cM\Delta T}{x_2}&e\left(\mp\frac{x_1{x_2}^2}{4c}\right)\int_{|v|\leq \Delta^{\varepsilon-1}}g(\Delta v)w_{2, \phi}(v) \\
    &\times e\left(\frac{vT}{\pi}+\frac{x_1{x_2}^2}{4c}(\pm1-\phi(v)^{-2}))\right)\dd v+O(T^{-A}).
    \end{split}
\end{equation}
Then we use the Mellin transform to analyse the $v$-integral, by the same proof as that of \cite[Lemma 8.2]{young2014weyl}, we obtain (\ref{eqn:Iphi=eint2}) and (\ref{eqn:U}) as claimed.
\end{proof}
We summarize the above discussion as the following lemma.
\begin{lemma}\label{lemma:Wproperty}
Let $C\geq 1$ and $X_1, X_2$ with $|X_1|,|X_2|\geq 1$
$x_1\asymp X_1$, $x_2\asymp X_2$ and $c\asymp C$,
$\sigma\in\{\pm\}$.
Set $X=\frac{|X_1|{X_2}^2}{C}>0$. We have
\[
W_{\pm}(x_1, x_2; c )\ll T^{-100},
\]
unless $|X_1|M\asymp |X_2|N$ and $X_1X_2>0$. Notice that
\begin{equation*}
    W_{\pm}(-x_1, -x_2; c )=\overline{W_{\pm}(x_1, x_2; c )},
\end{equation*}
we may assume that $X_1$, $X_2$ are all positive.\\
If $\frac{\sqrt{M}N}{C}\gg T^{1-\varepsilon}\Delta$ and
${X_1}{X_2}^2\asymp M^{\frac{1}{2}}N$,
we have
\begin{equation}
W_{+}(x_1, x_2; c )=\frac{cM\Delta}{x_2}e\left(-\frac{x_1{x_2}^2}{4c}\right)\int_{|u|\ll U_{+}}\lambda_{X, T}(u)\left(\frac{x_1{x_2}^2}{4c}\right)^{iu}\dd u+O(T^{-10}).
\end{equation}
where $\lambda_{X,T}$ depends on $X$, $T$, such that $\lambda_{X,T}(u)\ll 1$, and
$U_{+}=T^2/X$.\\
Or if $\frac{\sqrt{M}N}{C}\asymp T$ and
${X_1}{X_2}^2\ll M^{\frac{1}{2}}N\Delta^{\varepsilon-3}$,
we have
\begin{equation}
W_{-}(x_1, x_2; c )=\frac{cM\Delta}{x_2}e\left(\frac{x_1{x_2}^2}{4c}\right)\int_{|u|\ll U_{-}}\lambda_{X, T}(u)\left(\frac{x_1{x_2}^2}{4c}\right)^{iu}\dd u+O(T^{-10}).
\end{equation}
where $\lambda_{X,T}$ depends on $X$, $T$,  such that $\lambda_{X,T}(u)\ll 1$, and
$U_{-}=X^{\frac{1}{3}}T^\frac{2}{3}$.\\
Otherwise,
\begin{equation}
W_{\pm}(x_1, x_2; c )\ll T^{-10}.
\end{equation}
\end{lemma}

\subsection{The arithmetic part}
According to definition of $G_\pm$ (\ref{eqn:G=}), by changing variable $a$ to $-a$, we get
\begin{equation}
    G_{+}(x_1, x_2; c)=G_{-}(-x_1, x_2; c),
\end{equation}
Changing the sign of $x_1,x_2$, we have
\begin{equation}\label{eqn:G=barG}
    G_{\pm}(-x_1, -x_2; c )=\overline{G_{\pm}(x_1, x_2; c )}.
\end{equation}
We define $G'_{\pm}(x_1,x_2;c)=G_{\pm}(x_1, x_2; c)e\left(\mp\frac{x_1{x_2}^2}{4c}\right)$. The above identity implies
\begin{equation}
    G'_{+}(x_1, x_2; c)=G'_{-}(-x_1, x_2; c)
\end{equation}
It suffices to consider the $+$ case.
Opening the Kloosterman sum, we get
\[
    G'_{+}(x_1,x_2; c)=\frac{1}{c^3}\sum_{a (c)}\sum_{ b(c)}\sideset{}{^*}\sum_{k (c)}e\left(\frac{a(k+x_1)+b^2\bar{k}+bx_2}{c}-\frac{x_1{x_2}^2}{4c}\right).
\]
The $a$-sum is not vanishing if $k\equiv -x_1 \bmod c$, which contributes $c$, This gives $(x_1,c)=1$,  and we obtain
\[
    G'_{+}(x_1,x_2; c)=\frac{\chi_{c}(x_1)}{c^2}\sum_{b \bmod c}e\left(\frac{-\bar{x}_1(b-\frac{x_1x_2}{2})^2}{c}\right),
\]
where $\chi_c$ is the principal Dirichlet character modulo $c$.
Next, we separate the $x_2$-term from the above sum. With some discussion on the parity, one derives
\begin{multline}\label{eqn: G=G_1,2}
    G'_{+}(x_1,x_2; c)=G_1(x_1; c)(1-\chi_2(x_1))\chi_2(x_2)\\
    +G_1(x_1; c)(1-\chi_2(x_2))+G_2(x_1; c)\chi_2(x_1)\chi_2(x_2),
\end{multline}
where $\chi_2$ is the principal Dirichlet character modulo $2$,
\[
    G_1(x; c)=\frac{\chi_c(x)}{c^2}\sum_{b \bmod c}e\left(\frac{-\bar{x}b^2}{c}\right),
\]
and
\[
    G_2(x; c)=\frac{\chi_c(x)}{c^2}\sum_{b \bmod c}e\left(\frac{-\bar{x}(b+\frac{1}{2})^2}{c}\right).
\]
This shows that $G_{\pm}$ is the linear combination of $G_{1}$ and $G_{2}$ with bounded coefficients independent of $c$.
By the orthogonality of additive characters,
we know that
\begin{equation}\label{eqn: Gsquare}
\sum_{c\asymp C}|G_i(x_1; c)|^2\ll \frac{1}{C^2},
\end{equation}
for any integer $x_1$ and index $i=1,2$. Moreover, by Weyl's difference method, one  can derive
\begin{equation}\label{eqn: Gupperbound}
 G'_{\pm}(x_1, x_2; c)\ll c^{-\frac{3}{2}+\varepsilon}.
\end{equation}

\subsection{Applying the large sieve inequality}
Applying Lemma \ref{lemma:Wproperty}, we have,
if $\frac{\sqrt{M}N}{C}\gg T^{1-\varepsilon}\Delta$,
\begin{equation}\label{eqn:S+asymp}
\begin{split}
    \mathcal{S}_{+}^{'}&(M, N; C)\ll  \mathop{\sum\sum}_{X_1, X_2 \, \text{dyadic}\geq 1\atop X_1M\asymp X_2N, X_1{X_2}^2\asymp M^{\frac{1}{2}}N}\frac{CM\Delta}{X_2}\\
    &\quad \times \int_{|u|\ll \frac{T^2C}{X_1{X_2}^2}}\left|\sum_{c\asymp C}\sum_{x_1\asymp X_1}\sum_{x_2\asymp X_2}G'_{+}(x_1,x_2; c)\lambda_{X,T}(u)\left(\frac{x_1{x_2}^2}{4c}\right)^{iu}\right|\dd u+O(T^{-3}).
\end{split}
\end{equation}
and
if $\frac{\sqrt{M}N}{C}\asymp T$,
\begin{equation}\label{eqn:S-asymp}
\begin{split}
    \mathcal{S}_{-}^{'}&(M, N; C)\ll  \mathop{\sum\sum}_{X_1, X_2 \, \text{dyadic}\geq 1\atop X_1M\asymp X_2N, X_1{X_2}^2\ll M^{\frac{1}{2}}N\Delta^{\varepsilon-3}}\frac{CM\Delta}{X_2}\\
    &\quad\times \int_{|u|\ll \frac{{X_1}^\frac{1}{3}{X_2}^\frac{2}{3}{T}^\frac{2}{3}}{{C}^\frac{1}{3}}}
    \left|\sum_{c\asymp C}\sum_{x_1\asymp X_1}\sum_{x_2\asymp X_2}G'_{-}(x_1,x_2; c)\lambda_{X,T}(u)\left(\frac{x_1{x_2}^2}{4c}\right)^{iu}\right|\dd u+O(T^{-3}).
\end{split}
\end{equation}
Otherwise $\mathcal{S}_{\pm}^{'}(M, N; C)\ll T^{-2}$.
By the conjugacy (\ref{eqn:W=barW}) and (\ref{eqn:G=barG}), we can assume the dyadic parameters $X_1, X_2\geq 1 $.
Using (\ref{eqn: G=G_1,2}), we can bound the integral by
\begin{equation}\label{eqn: S_int_bound}
    \sup_{j=1,2}\int_{|u|\ll U_{\pm}}\left|\sum_{c\asymp C}\sum_{x_1\asymp X_1}\sum_{x_2\asymp X_2}G_{j}(x_1; c)a(x_1)b(x_2)\lambda_{X,T}(u)\left(\frac{x_1{x_2}^2}{4c}\right)^{iu}\right|\dd u,
\end{equation}
where $U_{+}=\frac{T^2C}{X_1{X_2}^2}$, $U_{-}=\frac{{X_1}^\frac{1}{3}{X_2}^\frac{2}{3}{T}^\frac{2}{3}}{{C}^\frac{1}{3}}$ and $a(x_1), b(x_2)$ are some bounded complex number.
Applying the Cauchy--Schwarz inequality, \eqref{eqn: S_int_bound} is bounded by
\begin{equation*}
    \begin{split}
    \sup_{j=1,2}\left(\int_{|u|\ll U_{\pm}}\left|\sum_{c\asymp C}\sum_{x_1\asymp X_1}G_{j}(x_1; c)a(x_1)\left(\frac{x_1}{c}\right)^{iu}\right|^2\dd u\right)^{\frac{1}{2}}\left(\int_{|u|\ll U_{\pm}}\left|\sum_{x_2\asymp X_2}b(x_2){x_2}^{2iu}\right|^2\dd u\right)^{\frac{1}{2}}
    \end{split}
\end{equation*}
By Lemma \ref{lemma:Largesieve} and Lemma \ref{lemma:Largesieve2} and (\ref{eqn: Gsquare}),
it is bounded by
\begin{equation*}
C^\varepsilon {X_1}^\varepsilon\left(\frac{(U_{\pm}+X_1C)(U_{\pm}+X_2)X_1X_2}{C^2}\right)^\frac{1}{2}.
\end{equation*}
Put the above bounds in (\ref{eqn:S+asymp}) and (\ref{eqn:S-asymp}). Note that $M, N\ll T^{1+\varepsilon}$ and $C\leq T^{B}$ for a fixed constant $B$.
In (\ref{eqn:S+asymp}), we have $X_1\asymp \frac{N}{\sqrt{M}}$ and $X_2\asymp {\sqrt{M}}$.
$\frac{\sqrt{M}N}{C}\gg T^{1-\varepsilon}\Delta$ implies
$U_{+}=\frac{T^2C}{X_1{X_2}^2}\asymp \frac{T^2C}{N{M}^\frac{1}{2}}\ll\frac{T^{1+\varepsilon}}{\Delta}$
and
$X_1C\asymp \frac{NC}{\sqrt{M}}\ll T^{\varepsilon}U_{+}\ll \frac{T^{1+\varepsilon}}{\Delta}$.
obviously, $X_2\asymp {\sqrt{M}}\ll \frac{T^{2+\varepsilon}}{N{M}^\frac{1}{2}}\ll T^{\varepsilon}U_{+}$
Therefore
\begin{equation}\label{eqn:S'+bound}
  \begin{split}
  \mathcal{S}_{+}^{'}(M, N; C)&\ll\Delta M\left(\frac{M}{N}\right)^{\frac{1}{2}}(U_{+}+X_1C)^{\frac{1}{2}}(U_{+}+X_2)^{\frac{1}{2}}T^\varepsilon\\
  &\ll T^{\varepsilon}\Delta \sqrt{MN}U_{+}\ll T^{1+\varepsilon} \sqrt{MN}.
  \end{split}
\end{equation}
In (\ref{eqn:S-asymp}), we have $X_1\ll T^{\varepsilon}\frac{N}{M^{\frac{1}{2}}\Delta^3}$ and $X_2\asymp \frac{X_1M}{N}$.
$\frac{\sqrt{M}N}{C}\asymp T$ implies
$M, N\gg T^{\varepsilon}C^2$
$U_{-}=\frac{{X_1}^\frac{1}{3}{X_2}^\frac{2}{3}{T}^\frac{2}{3}}{{C}^\frac{1}{3}}\ll\frac{T^{1+\varepsilon}}{\Delta}$ and
$X_1C\ll T^{\varepsilon}\frac{NC}{M^{\frac{1}{2}}\Delta^3}\asymp \frac{N^2}{T\Delta}\ll T^{\varepsilon}U_{-}$,
$X_2\ll T^{\varepsilon}U_{-}$.
Therefore
\begin{equation}\label{eqn:S'-bound}
  \begin{split}
  \mathcal{S}_{-}^{'}(M, N; C)&\ll T^\varepsilon \Delta \sup_{X_1\ll T^{\varepsilon}\frac{N}{M^{\frac{1}{2}}\Delta^3} \atop X_2\asymp \frac{X_1M}{N} }M\left(\frac{X_1}{X_2}\right)^{\frac{1}{2}}(U_{-}+X_1C)^{\frac{1}{2}}(U_{-}+X_2)^{\frac{1}{2}}T^\varepsilon\\
  &\ll T^{\varepsilon}\Delta\sqrt{MN}U_{-}\ll T^{1+\varepsilon} \sqrt{MN}.
  \end{split}
\end{equation}

\subsection{The remaining case}
 We now give the estimate for $ \mathcal{S}_{\pm}^{''}(M, N, C)$ which contributes the error term in the main theorem \ref{thm:moment}. Recall that
 $ \mathcal{S}_{\pm}^{''}(M, N, C)$ is the sum with restrictions
\label{subsection: remaining-case}
\begin{equation}\label{assump:XM/C XN/C small}
|x_1|M/C<T^{\varepsilon}\quad \text{and} \quad |x_2|N/C< T^{\varepsilon}.
\end{equation}
According to Lemma \ref{lemma:H+} and (\ref{eqn:W}), $W_{+}(x_1, x_2; c)$  is small unless $\sqrt{M}N/C\gg \Delta T^{1-\varepsilon}$.
If $x_1\neq 0$ or $x_2\neq 0$, then (\ref{assump:XM/C XN/C small}) implies $\sqrt{M}N/C\ll T^{\varepsilon}\max\{N/\sqrt{M}, \sqrt{M}\}\ll T^{1+\varepsilon}$. Taking $\Delta\gg T^{\varepsilon}$, we have $W_{+}(x_1, x_2; c)$ is small. If $x_1=x_2= 0$, we have $G_{+}(0, 0; c)=0$.
Thus
 \begin{equation}\label{eqn:S''+bound}
    \begin{split}
      \mathcal{S}_{+}^{''}(M, N; C)\ll  T^{-10}.
    \end{split}
  \end{equation}
\par
Using Lemma \ref{lemma:H-}, we have
 $W_{-}(x_1, x_2, c)$  is small unless $\sqrt{M}N/C\asymp T$.
By trivial estimate for \eqref{eqn:W},
\begin{equation}\label{W- bound2}
    W_{-}(x_1, x_2; c)\ll  MNT^{1+\varepsilon}.
    \end{equation}
By (\ref{eqn: Gupperbound}), we have
\begin{equation}\label{eqn:S''-bound}
    \mathcal{S}_{-}^{''}(M, N; C)\ll T^{\varepsilon}C\frac{C}{N}\frac{C}{M}TMNC^{-3/2+\varepsilon}\ll T^{1+\varepsilon}\sqrt{MN}.
 \end{equation}
Combining (\ref{eqn:S=S_1+S_2}), (\ref{eqn:S'+bound}), (\ref{eqn:S'-bound}), (\ref{eqn:S''+bound}) and (\ref{eqn:S''-bound}), we finish the proof of Proposition \ref{proposition:S}.

\section{Proof of Theorem \ref{thm:moment_holo}}\label{sec:holo}
The proof of Theorem \ref{thm:moment_holo} is roughly repeating the progress of proving Theorem \ref{thm:moment}.
We denote the mixed moment in Theorem \ref{thm:moment_holo} by $\mathcal{M}_{K, \Delta}$.
Let $f\in H_{4k}$,
we recall the approximate functional equations for central value of $L$-functions (see \cite[Theorem 5.3]{IwaniecKowalski2004analytic})
\begin{equation}\label{AFE3}
L(1/2, f)=2\sum_{n\geq 1}\frac{\lambda_f(n)}{n^{1/2}}\frac{1}{2\pi i}\int_{(2)}\frac{\Gamma(y+2k)}{\Gamma(2k)}\frac{e^{y^2}\dd y}{(2\pi n)^y y},
\end{equation}
and (see \cite[Lemma 2.1 (2.2), (2.6)]{Khan2012Nonvanishing})
\begin{equation}\label{AFE4}
L(1/2, \sym^2 f)=\sum_{n\geq 1}\frac{\lambda_f(n^2)}{n^{1/2}}V_{4k}(n),
\end{equation}
and
\begin{equation}
V_{4k}(n)=\log\frac{4k}{n}+C+O\left(\frac{n}{k}\right),
\end{equation}
where the constant $C=2\gamma-\frac{3\log \pi}{2}+\frac{\psi(3/4)}{2}-2\log 2=\frac{3}{2}\gamma-\frac{3\log\pi}{2}-\frac{5\log 2}{2}+\frac{\pi}{4}.$\par

The Petersson trace formula says
\begin{equation}\label{PeterssonTF}
\sum_{f\in H_{4k}}w_f\lambda_f(n)\lambda_f(m)=\delta(n,m)+2\pi \sum_{c\geq 1}\frac{S(n, m; c)}{c}J_{4k-1}\left(\frac{4\pi\sqrt{nm}}{c}\right).
\end{equation}
Let $h_{K,\Delta}(k)=h\left(\frac{4k-K-1}{\Delta}\right)$.
Applying the above approximate functional equations and Petersson trace formula (\ref{PeterssonTF}), we have
\begin{equation}\label{eqn:M=D+R}
\mathcal{M}_{K, \Delta}=\mathcal{D}+ \mathcal{R},
\end{equation}
where $\mathcal{D}$ and  $\mathcal{R}$ represents the contribution of the diagonal term and the off-diagonal term respectively.
The diagonal term is
\begin{equation}\label{eqn:Diagonal_holo}
\begin{split}
\mathcal{D}
&=2\sum_{k\geq 1}h_{K,\Delta}(k)\sum_{n\geq 1}\sum_{m\geq 1}\frac{\delta(n^2, m)}{n^{1/2}m^{1/2}}V_{4k}(n) \frac{1}{2\pi i}\int_{(2)}\frac{\Gamma(y+2k)}{\Gamma{(2k)}}\frac{e^{y^2}\dd y}{(2\pi m)^y y}\\
&=2\sum_{k\geq 1}h_{K,\Delta}(k) \frac{1}{2\pi i}\int_{(2)}\frac{\Gamma(y+2k)}{\Gamma{(2k)}}\sum_{n\geq 1}\frac{V_{4k}(n)}{n^{3/2+2y}}\frac{e^{y^2}\dd y}{(2\pi)^y y}\\
&=2\sum_{k\geq 1}h_{K,\Delta}(k) \frac{1}{2\pi i}\int_{(2)}\frac{\Gamma(y+2k)}{\Gamma{(2k)}}\sum_{n\geq 1}\frac{1}{n^{3/2+2y}}\left(\log\frac{4k}{n}+C+O\left(\frac{n}{k}\right)\right) \frac{e^{y^2}\dd y}{(2\pi)^y y}.\\
\end{split}
\end{equation}
For the $O$-term, we move the integral line to $\Re(y)=\frac{1}{4}+\varepsilon$. The corresponding integral contributes at most
\begin{equation*}
\int_{-\infty}^{\infty}\frac{\Gamma(\frac{1}{4}+\varepsilon+2k)}{\Gamma{(2k)}}\zeta(1+2\varepsilon)\frac{e^{1-t^2}\dd t}{k(1+|t|)}\ll k^{-3/4+\varepsilon}.
\end{equation*}
And the contribution from $O$-term in $\mathcal{D}$ is $O(\Delta K^{-3/4+\varepsilon})$.  Thus, we have
\begin{multline}
\mathcal{D}
=2\sum_{k\geq 1}h_{K,\Delta}(k)\frac{1}{2\pi i}\int_{(2)}\frac{\Gamma(y+2k)}{\Gamma{(2k)}}\\
\times\left(\zeta(\frac{3}{2}+2y)\log 4k+ C\zeta(\frac{3}{2}+2y)+\zeta'(\frac{3}{2}+2y)\right)
\frac{e^{y^2}\dd y}{(2\pi)^y y}+O(\Delta K^{-3/4+\varepsilon}).\\
\end{multline}
Shifting the integral line to $\Re(y)=-1/4+\varepsilon$, cross the simple pole at $y=0$, and the residue is
$
\zeta(\frac{3}{2})\log 4k+C\zeta(\frac{3}{2})+\zeta'(\frac{3}{2}).
$
and the shifted integral is bounded by $O(k^{-1/4+\varepsilon})$.
Thus
\begin{multline}\label{eqn:Diagonal_holo asymp}
        \mathcal{D}=2\sum_{k\geq 1}h_{K,\Delta}(k)\left( \zeta(\frac{3}{2})\log 4k+C\zeta(\frac{3}{2})+\zeta'(\frac{3}{2})\right)+O(\Delta K^{-\frac{1}{4}+\varepsilon})\\
        =2\zeta(\frac{3}{2})\sum_{k\geq 1}h_{K,\Delta}(k)\log k +2\zeta(\frac{3}{2})\left(\frac{3}{2}\gamma-\frac{3\log(2\pi)}{2}+\log 2+\frac{\pi}{4}+\frac{\zeta'(3/2)}{\zeta(3/2)}\right)\sum_{k\geq 1}h_{K,\Delta}(k)\\+O(\Delta K^{-\frac{1}{4}+\varepsilon}).
\end{multline}

The contribution of the off-diagonal term $\mathcal{R}$ can also be reduced to bounding the sum
\begin{equation}\label{eqn:holo_S}
    \mathcal{S}_{\holo}(M, N; C)=\sum_{m\geq1}
  \sum_{n\geq1} w_M\left(\frac{m}{M}\right) w_N\left(\frac{n}{N}\right) \sum_{c\asymp C}\frac{S(n^2, m;c)}{c} H^{\holo}\left(\frac{4\pi\sqrt{m}n}{c}\right),
\end{equation}
where
\begin{equation}\label{eqn:holo_H}
    H^{\holo}(x)=\sum_{k\geq 1}h\left(\frac{4k-K-1}{\Delta}\right)J_{4k-1}(x).
\end{equation}
By \cite[pp. 85--86]{IwaniecTopics}, we have
\begin{equation}
    H^{\holo}(x)=\sum_{n\equiv 1(\mod 4)}h\left(\frac{n-K}{\Delta}\right)J_{n}(x)
    =\frac{1}{4}\int_{-\infty}^{\infty}\hbar(t)c(t)\dd t,
\end{equation}
where
\[
    \hbar(t)=\int_{-\infty}^{\infty}h\left(\frac{y-K}{\Delta}\right)e(ty)\dd y,
\]
\[
    c(t)=-2i\sin(x\sin(2\pi t))+2\sin(x\cos(2\pi t)).
\]
By direct calculation, $\hbar(t)=\Delta e(tK)\hat{h}(-\Delta t)$. Let $g(x)=\hat{h}(-\Delta x/2\pi)/4$, we have
\begin{equation}
    H^{\holo}(x)=\Delta \int_{-\infty}^{\infty}e\left(\frac{vK}{\pi}\right)g(\Delta v)(-2i\sin(x\sin(2\pi t))+2\sin(x\cos(2\pi t)))\dd t.
\end{equation}
where $g(x)$ satisfying $g^{j}(x)\ll_{j, A}(1+|x|)^{-A}$.
Thus $H^{\holo}(x)\cdot K$  can be essentially decompose into four pieces of the form (\ref{eqn:H_0}).
Similarly, the trivial estimate can bound the sum in the case of $c$ large well and
\begin{equation}\label{eqn:mathcalR_bound}
    \mathcal{R}\ll K^{\varepsilon}\sup_{M,N\ll K^{1+\varepsilon}, C\ll K^B}\frac{|\mathcal{S}_{\holo}(M, N; C)|}{\sqrt{MN}}+O(K^{-10}).
\end{equation}
for some fixed large $B>0$.
Remark \ref{remark:Pro. holds for H_0} implies that Proposition \ref{proposition:S} holds for replacing $H^{\pm}$ by $H^{\holo}(x)\cdot K$. Thus
\begin{equation}\label{eqn:S_holo bound}
    \mathcal{S}_{\holo}(M, N; C)\ll \sqrt{MN}K^{\varepsilon}.
\end{equation}
Theorem \ref{thm:moment_holo} follows upon combining (\ref{eqn:M=D+R}), (\ref{eqn:Diagonal_holo asymp}), (\ref{eqn:mathcalR_bound}) and (\ref{eqn:S_holo bound}).
\section*{Acknowledgements}
The authors would like to thank Jesse Thorner for his very helpful comments and suggestions.
This material is based upon work supported
by the Swedish Research Council under grant no. 2021-06594
while Bingrong Huang was in residence at Institut Mittag-Leffler in Djursholm, Sweden during the spring semester of 2024.


\end{document}